\documentclass[preprint,12pt]{elsarticle}

\usepackage{lipsum}
\usepackage{amsfonts}
\usepackage{graphicx}
\usepackage{epstopdf}
\usepackage{algorithm} 
\usepackage[noend]{algpseudocode}
\usepackage{bm}
\usepackage{algpseudocode}
\usepackage{amsmath}
\usepackage{lineno,hyperref,dsfont}
\usepackage{amscd}
\usepackage{mathrsfs}
\usepackage{stmaryrd}
\usepackage{color}
\usepackage{ifthen}
\usepackage{subfigure}
\usepackage{epsfig,wrapfig,epic}
\usepackage{url}
\usepackage{dsfont}
\usepackage{marvosym}
\usepackage{color}

\usepackage[none]{hyphenat}
\ifpdf
  \DeclareGraphicsExtensions{.eps,.pdf,.png,.jpg}
\else
  \DeclareGraphicsExtensions{.eps}
\fi

\newcommand{\rank}{\textsf{rank}}

\newcommand{\Jac}{\mathcal{J}}
\newcommand{\Sing}{\mathfrak{S}}
\newcommand{\LD}{\textsc{ld}}

\newcommand{\IDAE}{\sc{idae}}
\newcommand{\VIE}{\sc{vie}}
\newcommand{\VIDE}{\sc{vide}}
\newcommand{\DAE}{\sc{dae}}
\newcommand{\IAE}{\sc{iae}} 
\newcommand{\IPDAE}{\sc{ipdae}}
\newcommand{\PDAE}{\sc{pdae}}
\newcommand{\HVT}{\sc{hvt}}
\newcommand{\DOF}{\sc{dof}}
\newcommand{\AP}{\sc{ap}}

\newcommand{\R}{\mathbb{R}}

\newcommand{\ie}{{i}.{e}.}
\newcommand{\eg}{{e}.{g}. }
\newcommand{\etc}{{etc}.}

\newtheorem{example}{Example}[section]
\newtheorem{remark}{Remark}[section]
\newtheorem{prop}{Proposition}[section]
\newtheorem{cor}{Corollary}[section]

\newtheorem{define}{Definition}[section]
\newtheorem{theorem}{Theorem}[section]
\newtheorem{lemma}{Lemma}[section]

\def\ctotDer{\textbf{D}}

\def\foorp{\hfill$\square$}
\begin{document}

\begin{frontmatter}

%% Title, authors and addresses

%% use the tnoteref command within \title for footnotes;
%% use the tnotetext command for theassociated footnote;
%% use the fnref command within \author or \address for footnotes;
%% use the fntext command for theassociated footnote;
%% use the corref command within \author for corresponding author footnotes;
%% use the cortext command for theassociated footnote;
%% use the ead command for the email address,
%% and the form \ead[url] for the home page:
%% \title{Title\tnoteref{label1}}
%% \tnotetext[label1]{}
%% \author{Name\corref{cor1}\fnref{label2}}
%% \ead{email address}
%% \ead[url]{home page}
%% \fntext[label2]{}
%% \cortext[cor1]{}
%% \affiliation{organization={},
%%             addressline={},
%%             city={},
%%             postcode={},
%%             state={},
%%             country={}}
%% \fntext[label3]{}

\title{Structural Analysis by Modified Signature Matrix for Integro-differential-algebraic Equations}

%% use optional labels to link authors explicitly to addresses:
%% \author[label1,label2]{}
%% \affiliation[label1]{organization={},
%%             addressline={},
%%             city={},
%%             postcode={},
%%             state={},
%%             country={}}
%%
%% \affiliation[label2]{organization={},
%%             addressline={},
%%             city={},
%%             postcode={},
%%             state={},
%%             country={}}

%% Group authors per affiliation:
\author[mymainaddress,mysecondaryaddress]{Wenqiang Yang}
%%\ead[url]{www.elsevier.com}

%% or include affiliations in footnotes:

\author[mymainaddress,mysecondaryaddress]{Wenyuan Wu\corref{mycorrespondingauthor}}
\cortext[mycorrespondingauthor]{Corresponding author}
\ead{wuwenyuan@cigit.ac.cn}

\author[mythirdaddress]{Greg Reid}
%%\ead[url]{www.elsevier.com}

\address[mymainaddress]{Chongqing Key Laboratory of Automated Reasoning and Cognition, Chongqing Institute of Green and Intelligent Technology, Chinese Academy of Sciences}
\address[mysecondaryaddress]{Chongqing School, University of Chinese Academy of Sciences}
\address[mythirdaddress]{Mathematics Department, University of Western Ontario}

\begin{abstract}

 Integro-differential-algebraic equations ({\IDAE})s are widely used in applications of engineering and analysis. When there are hidden constraints in an {\IDAE}, structural analysis is necessary. But if derivatives of dependent variables appear in their integrals, the existing definition of the signature matrix for an {\IDAE} cannot be satisfied. Moreover, if an {\IDAE} has a singular Jacobian matrix after structural analysis by the $\Sigma$-method, improved structural analysis methods are proposed to regularize it. However, the optimal value of an {\IDAE} may be negative which can not ensure the termination of the regularization. Furthermore, overestimation of the signature matrix may also lead to failure of its structural analysis.

In this paper, firstly, we redefine the signature matrix and introduce a definition of the degree of freedom for {\IDAE}s. Thus, the termination of improved structural analysis methods can be guaranteed. Secondly, the detection method by points is proposed to deal with the problem of overestimation of signature matrix. Thirdly, the embedding method has proved to suitable for structural unamenable {\IDAE}s, including those types that arise from symbolic cancellation and numerical degeneration. 
Finally, the global numerical method is applied to an example of two-stage drive system which can help to find all solutions for {\IDAE}s by witness points. Hopefully, through the example of pendulum curtain, the approach for {\IDAE}s proposed in this paper can be applied to integro-partial-differential-algebraic equations ({\IPDAE})s.
%Finally, a global numerical method is utilized to find all solutions for {\IDAE}s by witness points. An example of two-stage drive system is used to demonstrate our method and its advantages. The application of the approaches for {\IDAE}s in this paper to {\IPDAE}S is promising. 
\end{abstract}

\begin{keyword}
%% keywords here, in the form: keyword \sep keyword
integro-differential-algebraic equation \sep signature matrix \sep degree of freedom \sep  structural method \sep witness point
%% PACS codes here, in the form: \PACS code \sep code
\PACS 0000 \sep 1111
%% MSC codes here, in the form: \MSC code \sep code
%% or \MSC[2008] code \sep code (2000 is the default)
\MSC 0000 \sep 1111
\end{keyword}
\end{frontmatter}

%% \linenumbers

%% main text
\section{Background}
%\subsection{Background}

In applications, {\IDAE}s may occur in the following cases. First of all, {\IDAE}s are often used to analyze dynamic changes during a interval of time or a distance, such as in electric circuits \cite{Jiang1999}, hydraulic circuits \cite{Nass2008}, chemical reactions \cite{Kafarov1999} and one-dimensional heat conduction \cite{Jumarhon1996}. Secondly, {\IDAE}s depend on the selection of dependent variables during modeling. When analyzing the change of current, according to Kirchhoff laws, capacitors correspond to the differential of the change of current, while inductors correspond to the integral of the change of current \cite{Richard2017}. Thirdly, the continuous-time PID controller \cite{Visioli2003} is widely used in control engineering which is an {\IDAE} system with three parameters to be determined.

 Let the independent variable $ t \in \mathbb{I}=[t_0,t_f] \subseteq \mathbb{R}$
and the unknown dependent variable $\bm{x} = (x_1(t),\cdots,x_n(t))$. Suppose
$\bm{x},\bm{x}^{(\bm{1})},\cdots,\bm{x}^{(\bm{\ell})}$ are vectors in $\mathbb{R}^n$, where $\bm{\ell}=(\ell_1,\cdots,\ell_n)$ is a $n$ dimensional non-negative integer vector, and $\bm{x}^{(\bm{\ell})}=({x}^{(\ell_1)}, \cdots, {x}^{(\ell_n)})$. ${x}^{(\ell_k)}$ is the $\ell_k$-th order derivative of ${x}_{k}(t)$, for $k=1,\cdots,n$, with $\bm{x}^{(\bm{0})} \equiv \bm{x}$.
Here we consider maps
$\bm{\phi}: \mathbb{I}\times \mathbb{R}^{ (\sum_{k=1}^{n}\ell_k+1)n}\rightarrow \mathbb{R}^m$ and $\bm{\varphi}: \mathbb{I}\times\mathbb{I}\times \mathbb{R}^{(\sum_{k=1}^{n}\ell_k+1)n}\times \mathbb{R}^{(\sum_{k=1}^{n}\ell_k+1)n}\rightarrow \mathbb{R}^m$ are given real analytic, where possibly $m \not = n$. An integro-differential-algebraic equation ({\IDAE}) can be described as follow:

\begin{equation}\label{eqn:IDAE}
\bm{\phi}(t,\bm{x}^{(\bm{\ell})}{(t)}) +\int_{t_0}^{t}{\bm{\varphi}(t,s,\bm{x}^{(\bm{\ell})}(s),\bm{x}^{(\bm{\ell})}{(t)})ds}= \bm{F}(t,\bm{x},...,\bm{x}^{(\bm{\ell})})= \bm{0}
\end{equation}

An {\IDAE} in Equation (\ref{eqn:IDAE}) consists of differential algebraic equation ({\DAE})'s part --- $\bm{\phi}(t,\bm{x}^{(\bm{\ell})}{(t)})$, noted as $\bm{\Phi}$, and integral algebraic equation ({\IAE})'s part --- $\int_{t_0}^{t}{\bm{\varphi}(t,s,\bm{x}^{(\bm{\ell})}(s),\bm{x}^{(\bm{\ell})}{(t)})ds}$, noted as $\bm{\Psi}$.
If $\bm{\varphi}(t,s,\bm{x}^{(\bm{\ell})}(s),\bm{x}^{(\bm{\ell})}{(t)})=\bm 0$, then $\bm{F}$ is a typical {\DAE} $\bm{\phi}(t,\bm{x}^{(\bm{\ell})}{(t)})=\bm 0$. If $\bm{\phi}(t,\bm{x}^{(\bm{\ell})}{(t)})=\bm 0$, then $\bm{F}$ is a typical {\IAE} $\int_{t_0}^{t}{\bm{\varphi}(t,s,\bm{x}^{(\bm{\ell})}(s),\bm{x}^{(\bm{\ell})}{(t)})ds}=\bm 0$.

 For linear time-varying {\IDAE} systems, the Laplace transform is widely used to solve most of them effectively \cite{Richard2017}. However, when dealing with an {\IDAE} system with a high index or nonlinear or singular Jacobian matrix, the Laplace transform will fail. Collocation methods are also a good choice for numerical solution of such {\IDAE}, but are only applicable to low index  ($\leq1$)\cite{Liang2019} or specific \cite{Pishbin2015} {\IDAE}s.
Thus, similar to {\DAE}s, it is necessary for an {\IDAE} to be processed by structural analysis.

Existing methods \cite{Shampine02,Vieira20011,Brown98,Leimkuhler91} for a {\DAE} to find an initial point on constraints can be applied directly for an {\IDAE} with on integrals at initial time $t=t_0$.  Moreover, the Homotopy method can find all initial points of a polynomial {\DAE} from each component through at least one witness point \cite{Yang2021}. And these witness points can help to detect not only the failure caused by a singular Jacobian matrix but also the failure caused by an incorrect signature matrix (see Example \ref{ex:1}).
%But sufficiently smooth can not always ensure the rank of Jacobian matrix in the neighborhood of initial point, especially for an {\IDAE}.
However, there is little research in this area.

Much work has been done on different indices \cite{Gear90,Liang2013,Liang2019} for structural analysis.  The differential index of this {\IDAE} can also be defined as the minimum number of differentiation to find all constraints. It is difficult to find hidden constraints directly {\cite{Shampine02}} for {\IDAE}s with high index ($\geq2$). Without index reduction, a numerical solution may drift off the true solution \cite{ZOLF2021}. The direct method {\cite{Gear88}} of index reduction is too complex to be applied. The indirect method for the $\Sigma$-method {\cite{Pryce01}} is based on the signature matrix and is efficient. To make it possible for {\IAE}s {\cite{Lamm2000}} and {\IDAE}s {\cite{Zolfaghari2019}}, the signature matrix had to be redefined. However, this definition of the signature matrix is still incomplete when it gives rise to an {\IDAE} with derivatives.

 An {\IDAE} is called ``\textbf{structurally amenable}" (S-amenable) if the $\Sigma$-method is applicable. Similarly, in a ``\textbf{structurally unamenable} (S-unamenable) {\IDAE}, the $\Sigma$-method may fail since the Jacobian matrix is singular after differentiation. Improved structural methods have been proposed to regularize the Jacobian matrices of S-unamenable {\DAE}s, including direct methods {\cite{Gerdts11,Iwata03,Murota95,Campbell93,Wu13}}. Indirect methods have also been proposed (the LC-method \cite{Tan17}, the embedding method \cite{Yang2021}, the ES-method \cite{Tan17}, the substitution method \cite{Taihei19} and the augmentation method \cite{Taihei19}). For {\IDAE}s, Bulatov \cite{Bulatov11} dealt with some linear explicit singular {\IDAE}s by using special properties of the matrix polynomials. Zolfaghari {\cite{ZOLF2021}} extended the LC-method and the ES-method to linear {\IDAE}s. But these methods also fail in the case of numerical degeneration (see Example \ref{ex:Degeneration}). Essentially, unlike for {\DAE}s with the non-negative optimal value of the signature matrix, the termination of improved structural methods can no longer be guaranteed  since the optimal value may be negative.

In this paper, our contributions can be summarized as:
\begin{itemize}
    \item We modify the definition of the signature matrix for {\IDAE}s with derivatives in their integrals. We also define the degree of freedom for {\IDAE}s and then it is conducive to the termination of improved structural methods.
    \item We propose a detection method to correct the overestimation of the signature matrix due to constraints been ignored.
    \item We extend the embedding method for {\IDAE}s that can help to restore full-rank Jacobian matrices without algebraic elimination.

\end{itemize}

\section{Structural Analysis by $\Sigma$-method for {\DAE}s}\label{s:sa}

The $\Sigma$-method is an efficient method for {\DAE}s, and this motivates on extension of it to {\IDAE}s. In this section, we provide a brief introduction to it to explain terminology and notation.

\begin{define}\label{def:LD}
Suppose that the $k$-th order of derivative of $x_j$ occurs in $F_i$, then the partial derivative $\partial F_i/ \partial x_j^{(k)}$ is not identical zero. The \textit{leading derivative} of an equation  $F_i=0$ with respect to
$x_j$ is denoted by $\LD(F_i,x_j)$ and is the highest order of derivative such that some $F_i\in \bm{F}$ depends on $x_j^{(k)}$ for some $k\in \mathbb{Z}$.
Thus, we construct an $n \times n$ signature matrix $\bm \sigma(\bm{F})=[\sigma_{i,j}]_{1 \leq i \leq n,1 \leq j \leq n}(\bm{F})$ of {\DAE}s $\bm{F}$ by Pryce \cite{Pryce01}:
\begin{equation}\label{signature_DAE}
    [\sigma_{i,j}](\bm{F}):= \left\{%
\begin{array}{ll}
    \hbox{ the order of $\LD(F_i,x_j(t))$}, \;\;\hbox{if $x_j(t)$ occurs in $F_i$;} \\
    -\infty, \;\; \hbox{otherwise.} \\
\end{array}%
\right.
\end{equation}
\end{define}
%\subsection{Structural Analysis by $\Sigma$-method}

Let $\bm{c} =
(c_1, ..., c_n )$ and $\bm{d} = (d_1, ..., d_n )$ be a dual optimal solution.
There must be a highest-value transversal ({\HVT}) of the signature matrix, denoted by $\sum\limits_{{(i,j)\in T}}{\sigma_{ij}}$, in which  $d_j - c_i = \sigma_{ij}$ for all $(i,j)\in T$, and $T$ is the set of indices of elements in different rows and columns corresponding to the maximum value. According to \cite{ZOLF2021}, the dual problem is equivalent to minimizing $\sum\limits_{{(i,j)\in T}}{\sigma_{ij}}=\sum{d_j}-\sum{c_i}$.

This can be formulated as an
assignment problem ({\AP}):
\begin{equation}\label{LPP}
   \delta(\bm{F}) \left\|%
\begin{array}{l}
    \hbox{Minimize~~} \delta = \sum d_j - \sum c_i,  \\
    \hbox{~~~~where~~} d_j -c_i \geq \sigma_{ij},  \\
    ~~~~~~~~~~~~~~c_i \geq 0 \\
\end{array}%
\right.
\end{equation}
Let $\delta(\bm{F})$ be the optimal value of the problem (\ref{LPP}).

Let $\ctotDer$ be the formal total derivative operator with respect to independent variable $t$:
\begin{equation}\label{eq:Dt}
\ctotDer = \frac{\partial}{\partial t} +  \sum_{k=0}^{\infty} \bm{x}^{(k+1)} \frac{\partial }{\partial \bm{x}^{(k)}}
\end{equation}

If we specify the  differentiation order for $F_i$ to be $c_i$, then $c_{i}\geq 0$, for $i = 1, \dots, n$. Then the  differentiation of $\bm{F}$  up to the order  $\bm{c}$ is
\begin{equation}\label{eq:DefFc}
\bm{F}^{(\bm{c})}= \{F_1, \ctotDer F_1,..., \ctotDer^{c_1}F_1\}\cup \cdots \cup  \{F_n, \ctotDer F_n,..., \ctotDer^{c_n}F_n\} =\ctotDer^{\bm{c}}\bm{F}
\end{equation}

 The number of equations of $\bm{F}^{(\bm{c})}$ is $n + \sum_{i=1}^{n} c_i$.

Without loss of generality, we
assume $c_1 \geq c_2 \geq \cdots \geq c_n$, and let $k_c = c_1$,
which is closely related to the \textit{index} of system $\bm{F}$ (see
\cite{Pryce01}). The $r$-th order derivative of $F_j$
with respect to $t$ is denoted by $F_j^{(r)}$. Then we can
partition $\bm{F}^{(\bm{c})}$ into $k_c + 1$ parts,
for $0\leq p \in \mathbb{Z}\leq k_c$ given by
\begin{equation}\label{eq:B_i}
    \bm{B}_p:= \{ F_j^{(p+c_j - k_c )} : 1\leq j \leq n, p+c_j - k_c
    \geq 0 \}.
\end{equation}

Here, we call $\bm{B}_{k_c}$ as the \textbf{top block} of $\bm{F}^{(\bm{c})}$ and $ \bm{F}^{(\bm{c}-1)} = \{\bm{B}_0,...,\bm{B}_{k_c-1}\}$ as the \textbf{constraints}.

Similarly, let $k_d = \max(d_j)$ and we can partition all the variables into $k_d+1$ parts:
\begin{equation}\label{eq:U_i}
   % \bm{X}^{(q)}:= \{ x_j^{(q+d_j - k_d )} : 1\leq j \leq n, q+ d_j - k_d  \geq 0 \}.
   \bm{X}^{(q)}:= \{ x_j^{(q+d_j - k_d )} : 1\leq j \leq n \}.
\end{equation}

Here, if $(q+ d_j - k_d) < 0$, it means it's integral of $x_j$ with respect to the independent variable $t$.

%The success of $\Sigma$-method depends on the Jacobian matrix. And there is a relationship on Jacobian  matrices of all blocks 

%\begin{prop}\label{prop:fullrank}
%Let $ \{\bm{\Jac}_i\}$ be the set of Jacobian
%matrices of $\{\bm{B}_i \}$. For any $0 \leq i < j \leq k_c$, $\bm{\Jac}_{i}$
%is a sub-matrix of $\bm{\Jac}_{j}$. Moreover, if $\bm{\Jac}_{k_c}$ has full
%rank, then any $\bm{\Jac}_{i}$ also has full rank.
%\end{prop}
%\begin{proof}
% Since $\bm{\Jac}_{k_c}$
%is $m\times m$ full rank matrix, its rows are linearly independent. Since $\bm{\Jac}_{i}$
%is a sub-matrix
%of $\bm{\Jac}_{k_c}$
%, we can assume it consists of the first $p$ rows and first $q$ columns of $\bm{\Jac}_{k_c}$. If $q = m$, then $\rank(\bm{\Jac}_{i}) = p$. If $q < m$, then the entries in its first $p$ rows and last
%$m-q$ columns must be $0$. So $\rank(\bm{\Jac}_{i}) = p$.
%More detail see \cite{WRI09}.
% \foorp
%\end{proof}

%Suppose $(t^*,\bm{X}^*)$ is a point satisfying the  constraints and
%$\bm{\Jac}_{k_c}$ has full rank at this point, then $\Sigma$-method is succeed.
%However, it fails if $\bm{\Jac}_{k_c}$ is still singular.

%In the rest of the paper, we usually suppress the subscript in $\bm{\Jac}_{k_c}$ so it becomes $\bm{\Jac}$ unless the subscript is needed.

\section{Structural Analysis for {\IDAE}s}\label{s:key}
 However,
for {\IDAE}s, the existing definition of the signature matrix has encountered great challenges, due to the derivatives in the {\IAE}'s part of the {\IDAE}. 

The improved structural methods for {\DAE}s to find hidden constraints are equivalent to decreasing its optimal value $\delta$. However, when it encounters to {\IDAE}s, there are several difficulties: (a) the scheme of the combinatorial relaxation framework \cite{Taihei19} in Phase $3$ is not applicable; (b) the optimal value may be negative in Example \ref{ex:Zolf_2}; (c) the termination of improved structural methods can not be ensured. 

 Thus, we should adapt several general definitions to erase these difficulties in structural analysis.
 
\subsection{Modified Signature Matrix for {\IDAE}s}\label{ss:SM}

 Unlike {\DAE}s or {\IAE}s, the signature matrix of {\IDAE}s must contain the information of both parts.

For the {\DAE} part of the {\IDAE}, we can easily construct an $n \times n$ signature matrix $\bm \sigma(\bm{\Phi})=[\sigma_{i,j}]_{1 \leq i \leq n,1 \leq j \leq n}(\bm{\Phi})$ according to Definition \ref{def:LD}. For the {\IAE}'s part, \cite{Lamm2000} gave an incomplete definition of its signature matrix without considering derivatives. In this section, we give a new the definition of the signature matrix for {\IDAE}s.

\begin{define}\label{def:v-somoothing}
Let $\bm{\Psi}=\int_{t_0}^{t}{\bm{\varphi}(t,s,\bm{x}^{(\bm{\ell})}(s),\bm{x}^{(\bm{\ell})}{(t)})ds}$  be sufficiently smooth, defined in Equation (\ref{eqn:IDAE}). For any $x_j$ of $\bm{x}$ and for some $t \in \mathbb{I}$, let $\upsilon_{i,j} \geq 1$ be the \textbf{smallest} integer of some $\varphi_i \in \bm{\varphi}$ for which
\begin{equation}\label{eq:v-smooth}
\frac{\partial}{\partial x_j}\left(\left.{\frac{\partial^{\upsilon_{i,j}-1}}{\partial t^{\upsilon_{i,j}-1}}\varphi_i(t,s,\bm{x}^{(\bm{\ell})}(s),\bm{x}^{(\bm{\ell})}{(t)})}\right|_{s=t}\right)\neq 0
\end{equation}

%This means $\bm{\Psi}$ is $\upsilon_{i,j}$-smoothing \cite{Lamm2000} with respect to $x_j$.

Let $\omega_{i,j} \geq 1$ be the \textbf{largest} integer for which
\begin{equation}\label{eq:w-smooth}
\frac{\partial}{\partial x_j}\left(\left.{\frac{\partial^{\omega_{i,j}-1}}{\partial t^{\omega_{i,j}-1}}\varphi_i(t,s,\bm{x}^{(\bm{\ell})}(s),\bm{x}^{(\bm{\ell})}{(t)})}\right|_{s=t}\right)\neq 0
\end{equation}

We say $\bm{\Psi}$ is $\omega_{i,j}$-integral with respect to $x_j$.

If Equation (\ref{eq:v-smooth}) does not hold for any integer $\upsilon_{i,j} \geq 1$, then we define $\upsilon_{i,j}=\infty$ which means $\infty$-smoothing. 

If Equation (\ref{eq:w-smooth}) does not hold for any integer $\omega_{i,j} \geq 1$, then we define $\omega_{i,j}=0$ which means $x_{j}$ does not occur in ${\varphi_i}$.
\end{define}

\begin{remark}\label{rem:smooth}
Note that Equation (\ref{eq:v-smooth}) implies that $\bm{\Psi}$ is $\upsilon_{i,j}$-smoothing \cite{Lamm2000} with respect to $x_j$. In particular, this definition also applies to while $x_{j}$ does not occur in ${\varphi_i}$.
\end{remark}

\begin{example}\label{ex:v,w}
Let $\varphi_i(t,s,\bm{x}(s))=x_2(s)^{2}+(t-s)x_1(s)$, where $\bm{x}(s)=(x_1(s),x_2(s),x_3(s))\neq \bm{0}$. Here, 

\begin{equation*}
    \begin{array}{ll}
    \varphi_i(t,t,\bm{x}(t))=x_2(t)^2 & \\
          \frac{\partial}{\partial x_1}\left(\varphi_i(t,t,\bm{x}(t))\right)= 0, & \frac{\partial}{\partial x_{2}}\left(\varphi_i(t,t,\bm{x}(t))\right)= 2\cdot x_2 \neq 0, \\
         \frac{\partial}{\partial x_1}\left(\left.\frac{\partial}{\partial t}\varphi_i(t,s,\bm{x}(s))\right|_{s=t}\right)=1\neq 0, & \frac{\partial}{\partial x_2}\left(\left.\frac{\partial}{\partial t}\varphi_i(t,s,\bm{x}(s))\right|_{s=t}\right)=0,\\
         \frac{\partial}{\partial x_1}\left(\left.\frac{\partial^{k}}{\partial t^{k}}\varphi_i(t,s,\bm{x}(s))\right|_{s=t}\right)= 0, & \frac{\partial}{\partial x_2}\left(\left.\frac{\partial^{k}}{\partial t^{k}}\varphi_i(t,s,\bm{x}(s))\right|_{s=t}\right)=0,~for~k\geq 2.
     \end{array}
\end{equation*}

Therefore, by Equation (\ref{eq:v-smooth}) and Remark \ref{rem:smooth}, $\bm{\Psi}$ is $2$-smoothing with respect to $x_1$ and $1$-smoothing with respect to $x_2$ and $\infty$-smoothing with respect to $x_3$.  Significantly, $x_2(t)$ appearing in $\varphi_i$ is a dependent variable with independent variable $t$, which indicates that its corresponding element in the signature matrix of the {\DAE} part of {\IDAE} should be considered by Definition \ref{def:LD}.

By Equation (\ref{eq:w-smooth}), $\bm{\Psi}$ is $2$-integral with respect to $x_1$ and $1$-integral with respect to $x_2$, and $0$-integral with respect to $x_3$.

\end{example}

Since $\sigma_{i,j}$ is the order of the highest derivative of variable $x_j(s)$ occurs in the $i$-th function \cite{Pryce01}, we can define an $n \times n$ signature matrix as follows:

\begin{define}[Signature Matrix for {\IAE} Part]\label{def:sm_iae}
 Consider an {\IDAE} of Equation (\ref{eqn:IDAE}), we define the signature matrix, as $n \times n$ matrix $\bm \sigma(\bm{\Psi})=[\sigma_{i,j}]_{1 \leq i \leq n,1 \leq j \leq n}(\bm{\Psi})$ of {\IAE} part:
\begin{equation}\label{signature_IAE}
    [\sigma_{i,j}](\bm{\Psi}):=  \left\{%
\begin{array}{ll}
\hbox{$\max(\beta_1, \beta_2)$,}\;\;\ \hbox{$x_j(t)$ and $x_j(s)$ occur in $\varphi_i$;}\\
\hbox{ $\beta_2$,}\;\;\ \hbox{ only $x_j(s)$ occurs in $\varphi_i$.}\\
\hbox{$-\infty$,}\;\;\ \hbox{ otherwise;}
\end{array}\right.
\end{equation}
Where $\beta_1$ is the order of $\LD(\varphi_i,x_j(t))$ and $\beta_2$ is the order of $\LD(\varphi_i,x_j(s))-\upsilon_{i,j}$.
\end{define}
When $\varphi_i$ does not contain the derivative of $x_j$, the order of $\LD(\varphi_i,x_j)$ is $0$, which is the same as the definition of \cite{Zolfaghari2019}.

%To sum up, since $\sigma_{i,j}$ is the order of the highest derivative of variable $x_j$ occurs in the $i$-th function \cite{Pryce01}, the $n \times n$ signature matrix as follows:

\begin{define}[Signature Matrix for {\IDAE}] \label{de:sm_idae}
  The $n \times n$ signature matrix $\bm \sigma(\bm{F})=[\sigma_{i,j}]_{1 \leq i \leq n,1 \leq j \leq n}(\bm{F})$ of {\IDAE} $\bm{F}$ of Equation (\ref{eqn:IDAE}) is defined as:
\begin{equation}\label{signature}
    [\sigma_{i,j}](\bm{F}):= \max_{i,j}{\left([\sigma_{i,j}](\bm{\Phi}),[\sigma_{i,j}](\bm{\Psi})\right)}
\end{equation}
\end{define}
Obviously, this is equivalent to the signature matrix defined by Zolfaghar in \cite{ZOLF2021} in the case that there is no derivative in {\IAE}s part.

%\begin{remark}\label{rem:combination}
%Since any element $[\sigma_{i,j}](\bm{\varphi})\leq 0$ in {\IAE} part, signature matrix of {\IDAE}s usually is the same as the {\DAE} part, unless the corresponding variable and its derivative does not appear in the {\DAE} part.
%\end{remark}

\begin{example}\label{ex:Zolf}
\sloppy{}
Consider the following {\IDAE} \cite{ZOLF2021} with dependent variables $x_1\left( t \right)$ and $x_2\left( t \right)$:
\[{\bm{F}}=\left\{\begin{array}{l}
e^{-x_1(t)-x_2(t)}-g_1(t)\\
\int_{t_0}^{t}{\left(x_1(s)+x_2(s)+\left( t-s \right) x_1(s)\cdot x_2(s)\right)ds} - g_2(t)
\end{array}\right.\]
Where $g_1(t)$ and $g_2(t)$ are given functions. Its Jacobian  \[\bm{\Jac}=\left(\begin{array}{cc}
-e^{-x_1(t)-x_2(t)} & -e^{-x_1(t)-x_2(t)} \\
1 & 1\\
\end{array}\right)\] is identically singular, with rank equals to $1$.

By Equation (\ref{eqn:IDAE}), we get $\bm{\Phi}=\left\{\begin{array}{l}
e^{-x_1(t)-x_2(t)}-g_1(t)\\
 - g_2(t)
\end{array}\right.$ and
$\bm{\Psi}=\left\{\begin{array}{l}
0\\
\int_{t_0}^{t}{\left(x_1(s)+x_2(s)+\left( t-s \right) x_1(s)\cdot x_2(s)\right)ds}
\end{array}\right.$

Then, it easily follows that $[\sigma_{i,j}](\bm{\Phi})= \left(
        \begin{array}{cc}
          0 & 0\\
          -\infty & -\infty \\
        \end{array}
      \right)$ by Equation (\ref{signature_DAE}), and
     $[\sigma_{i,j}](\bm{\Psi})= \left(
        \begin{array}{cc}
         -\infty & -\infty\\
          -1 & -1 \\
        \end{array}
      \right)$ by Equation (\ref{signature_IAE}). So
      $\bm{\omega}= \left(
        \begin{array}{cc}
          0 & 0\\
          2 & 2 \\
        \end{array}
      \right)$ by Equation (\ref{eq:w-smooth}).
Thus, $[\sigma_{i,j}](\bm{F})= \left(
        \begin{array}{cc}
          0 & 0\\
          -1 & -1 \\
        \end{array}
      \right)$ by Equation (\ref{signature}).
\end{example}

%\subsection{Framework for Improved Structural Methods}\label{ssec:improved}

%Improved structural methods are based on the combinatorial relaxation framework in \cite{Taihei19}. The regularization of these methods are to find hidden constraints, which lead to the decrease of {\DOF}. And the termination of their combinatorial relaxation framework depends on the existence of the solution, which means $DOF\geq 0$. So we can modify the combinatorial relaxation framework as follows:
%\begin{description}
 % \item[Phase $1$.] Compute the solution ($\bm{c}$,$\bm{d}$) of {\AP} problem $\delta(\bm{F})$. If there is no solution, the {\IDAE} do not admit perfect matching, and the algorithm ends with failure.
 % \item[Phase $2$.] Determine whether $\bm{\Jac}_{k_c}$ is identically singular or not. If not, the method returns $\bm{F}^{(\bm{c})}$ and halts.
 % \item[Phase $3$.] Construct an new {\IDAE} $\bm{G}$, such that its solution space in $\bm{x}$ dimension is the same as {\IDAE} $\bm{F}$ and $0\leq DOF(\bm{G}) < DOF(\bm{F})$. Then go to Phase $1$.
%\end{description}

%Phase $2$ above is only to check for symbolic cancellation.

\subsection{The Degree of Freedom of {\IDAE}s}\label{ss:dof}
Finding hidden constraints is essential to minimizing the degree of freedom {\DOF} of an {\IDAE} $\bm{F}$. The termination of improved structural methods depends on the existence of the solution, which implies $DOF\geq 0$. The definition of general form of {\DOF} is as follows:

\begin{define}[{\DOF} for {\IDAE}]\label{def:DOF}
Let a system $\bm{F}$ contains $m$ equations and $n$ dependent variables, {\DOF} of $\bm{F}$ is ${DOF}(\bm{F}):= n -\rank (\bm{\Jac})$ which determines the existence of the solution. Without redundant equations, then ${DOF}(\bm{F}):= n - m$.
\end{define}

There is a relationship between {\DOF} and optimal value, which can help to deduce the {\DOF} of {\IDAE}s directly.

\begin{prop}\label{prop:DOF}
Let $(\bm{c},\bm{d})$ be the optimal solution of Problem (\ref{LPP}) for a given {\IDAE} $\bm{F}$. And $x_j$ is $\omega_{i,j}$-integral in $\varphi_i$ of $F_i$. Then ${DOF}(\bm{F}):= \delta (\bm{F})+\sum\limits_{j}{\max\limits_{i}{\omega_{i,j}}}$.
\end{prop}

\begin{proof}
Since any $x_j$ in $\varphi_i$ of $F_i$ is $\omega_{i,j}$-integral, there must be a primitive function with respect to the dependent variable $x_j$, whose $\omega_{i,j}$-th derivative with respect to  the independent variable $t$ is $x_j$. Thus there are $\omega_{i,j}$ dependent variables related to the integral of $x_j$ in $\varphi_i$ of $F_i$. Hence, there are $\max\limits_{i}{\omega_{i,j}}$ dependent variables related to the integral of $x_j$ in $\bm{F}$. Since $\bm{F}^{(\bm{c})}$ is the differentiation of $\bm{F}$, there are also $\max\limits_{i}{\omega_{i,j}}$ dependent variables related to the integral of $x_j$ in $\bm{F}^{(\bm{c})}$.

Obviously, the derivatives in  $\bm{F}^{(\bm{c})}$ are $\bm{x^{(\leq{d})}}$. Assume there are $n$ equations in $\bm{F}$. There must be $n + \sum\limits_{j}{(d_j+\max\limits_{i}{\omega_{i,j}})}$ dependent variables
and $n + \sum\limits_{i}{c_i}$ equations in $\bm{F}^{(\bm{c})}$.
And there must be $n + \sum\limits_{j}{(d_j+\max\limits_{i}{\omega_{i,j}})}-\sum\limits_{i}{c_i}$ dependent variables and $n $ equations  in $\bm{F}$. By Definition \ref{def:DOF}, $DOF (\bm{F})=DOF (\bm{F}^{(\bm{c})})$. Moreover, since $\delta (\bm{F})=\sum\limits_{j}{d_j}-\sum\limits_{i}{c_i}$, then
$$DOF (\bm{F})=\delta (\bm{F})+\sum\limits_{j}{\max\limits_{i}{\omega_{i,j}}}$$.

\foorp
\end{proof}

In the special case of a {\DAE} $\bm{F}$, we have  $\omega_{i,j}=0$. Then the  {\DOF} of $\bm{F}$ is $\delta (\bm{F})$, this is the same as the definition of {\DOF} in \cite{Tan17}.

Since the optimal value of $\bm{F}$ is limited to square systems, the {\DOF} of non-square systems $\bm{F}^{(\bm{c})}$ should be extended.

\begin{prop}\label{define_delta1}
Let an {\IDAE} $\bm{F}$  consist of two blocks $\bm{A}$ and $\bm{B}$,  where $\bm{F}$ contains $p$ equations and $n$  dependent variables $p\geq n$, and suppose the signature matrix of $\bm{A}$ is an $n\times n$ square matrix. So $\bm{B}$ contains the remaining $(p-n)$ equations. Let $DOF(\bm{A})$ be the degree of freedom of $\bm{A}$'s signature matrix. Then  $DOF(\bm{F}) = DOF(\bm{A}) - \#eqns(\bm{B})$, where $\#eqns(\bm{B})$ is the number of equations in $\bm{B}$.
%Meanwhile, $\delta(\bm{F})$  also equals the {\DOF} \cite{Tan17} of $\bm{F}$, which equals the number of dependent variables minus the number of equations in the prolongation of $\bm{F}$.
\end{prop}

\begin{proof}
Since the set of dependent variables of block $\bm{B}$ is a subset of the dependent variables of block $\bm{A}$, $DOF(\bm{A})=\#vars(\bm{A})-\#eqns(\bm{A})$ and $DOF(\bm{F})=\#vars(\bm{A})-\#eqns(\bm{F})$, where $\#vars(\bm{A})$ is the number of dependent variables in $\bm{A}$. Since there are no redundant equations, $\#eqns(\bm{F})=\#eqns(\bm{A})+\#eqns(\bm{B})$. Hence $DOF(\bm{F}) = DOF(\bm{A}) - \#eqns(\bm{B})$.
 \foorp
\end{proof}

%In the case of a square signature matrix of a {\DAE} $\bm{F}$, we have $\#eqns=0$, and the extended definition of $\delta(\bm F)$ is equivalent to the original definition.

%\begin{prop}\label{prop:DOF_ex}
%Let $(\bm{c},\bm{d})$ be the optimal solution of Problem (\ref{LPP}) for a given  {\IDAE} $\bm{F}$. Then $DOF (\bm{F})=DOF (\bm{F}^{(\bm{c})})=\delta (\bm{F})+\sum\limits_{j}{\max\limits_{i}{\omega_{i,j}}}$.

%\end{prop}

%\begin{proof}
%For a prolonged {\DAE} system $\bm{F}^{(\bm{c})}=\{\bm{B}_{k_c},\bm{F}^{(\bm{c}-1)}\}$, the signature matrix of the top block $\bm{B}_{k_c}$.

%We construct a pair $(\hat{\bm{c}},\hat{\bm{d})}$, for $i=1,\cdots,n$ and $ j=1,\cdots,n$, $\hat{c}_{i}=0$ and $\hat{d}_{j}= d_{j}$. Since $(\bm{c},\bm{d})$ is the optimal solution for $\bm{F}$, and $\bm{B}_{k_c}$ is the top block of $\bm{F}^{(\bm{c})}$, it follows that $(\hat{\bm{c}},\hat{\bm{d})}$ is the optimal solution of $\bm{B}_{k_c}$, $\delta(\bm{B}_{k_c})= \sum\limits_{j}{d_j}$. Such that $DOF(\bm{B}_{k_c})=n+\sum\limits_{j}{(d_j+\max\limits_{i}{\omega_{i,j}})}-n =\delta(\bm{B}_{k_c})+\sum\limits_{j}{\max\limits_{i}{\omega_{i,j}}}$.

%By Definition \ref{define_delta1}, and using $\#eqns(\bm{F}^{(\bm{c}-1)})=\sum c_i$, we obtain $$DOF (\bm{F}^{(\bm{c})})= DOF(\bm{B}_{k_c}) -\#eqns(\bm{F}^{(\bm{c}-1)}) =\delta (\bm{F}) +\sum\limits_{j}{\max\limits_{i}{\omega_{i,j}}} \,.$$
%\foorp
%\end{proof}

\begin{example}\label{ex:Zolf_2}
 \sloppy{}
Consider the {\IDAE} given in Example \ref{ex:Zolf} which has no constraints. The structural information obtained by the $\Sigma$-method is that the dual optimal solution is $\bm{c}=(0,1)$ and $\bm{d}=(0,0)$. Then the {\DOF} of this {\IDAE} by Proposistion \ref{prop:DOF} is $DOF(\bm{F})=\delta(\bm{F})+\sum\limits_{j}{\max\limits_{i}{\omega_{i,j}}}=-1+4=3$.

\end{example}

%\subsection{Framework for Improved Structural Methods}\label{ssec:improved}
%When Jacobian matrix is singular, the structural methods (\ie Pryce method) fail. Then we need to involve an improved structural method to regular the Jacobian matrix.
%Generally, improved structural methods are based on the combinatorial relaxation framework in \cite{Taihei19} as follows:
%\begin{description}
 % \item[Phase $1$.] Compute the solution ($\bm{c}$,$\bm{d}$) of {\AP} problem $\delta(\bm{F})$. If there is no solution, the {\IDAE} do not admit perfect matching, and the algorithm ends with failure.
 % \item[Phase $2$.] Determine whether $\bm{\Jac}_{k_c}$ is identically singular or not. If not, the method returns $\bm{F}^{(\bm{c})}$ and halts.
  %\item[Phase $3$.] Construct an new {\IDAE} $\bm{G}$, such that its solution space in $\bm{x}$ dimension is the same as {\IDAE} $\bm{F}$ and $0\leq \delta(\bm{G}) < \delta(\bm{F})$. Then go to Phase $1$.
%\end{description}

%Phase $2$ above is only to check for symbolic cancellation.

\section{Detection Method for Incorrect Signature Matrix}\label{s:detect}

  Unfortunately, the $\Sigma$-method updated with the definitions in section \ref{s:key} may fail to overestimate some ``true" $\sigma_{i,j}$. That will yield incorrect optimal solutions and hidden constraints \cite{Tan17}. The case of an incorrect signature matrix is a typical case of such failure. In that case, the function corresponding to $\sigma_{i,j}$, {\ie} to coefficient and Equation (\ref{eq:v-smooth}), {\etc}, may be vanish on the constraints of {\IDAE}. This occours in Example \ref{ex:1}, and may lead to incorrect optimal solutions.

\begin{example}\label{ex:1}
Consider the following {\IDAE} with dependent variables $x(t)$ and $y(t)$:
\begin{equation}\label{eq:ex1}
\bm{F} = \left\{y\left( t \right) -  \dot{x}\left( t \right) , x\left( t \right) + \int_{t_0}^{t} {\left(t-s\right)\cdot\left(\frac{y\left( s \right)}{2} -  \dot{x}\left( s \right)\right)\cdot y\left( s \right)}ds \right\}.
\end{equation}
Then 
$$
 [\sigma_{i,j}](\bm{F})=\begin{array}{ccc}
 & \begin{array}{cc} ~x & ~y \end{array} & c_i\\
 & \left(\begin{array}{cc}
   ~1 & ~0\\
     ~0 & \color{blue}\bm{-2}
 \end{array}\right) & \begin{array}{c}
   0\\
   \color{blue}\bm{2}
 \end{array}\\ d_j&\begin{array}{cc}
   ~\color{blue}\bm {2} & ~0
 \end{array} &
\end{array}   \color{red}\bm{\times}
$$

$$
 [\sigma_{i,j}](\bm{F})=\begin{array}{ccc}
  & \begin{array}{cc} x & ~~~y \end{array} & c_i\\
 & \left(\begin{array}{cc}
   ~~1 & ~0\\
    ~~0 & \color{blue}\bm{-\infty}
 \end{array}\right) & \begin{array}{c}
   0\\
   \color{blue}\bm {1}
 \end{array}\\ ~~d_j&\begin{array}{cc}
  \color{blue}\bm {1} & ~~~0
 \end{array} &
\end{array}   \color{red}\bm{\checkmark}
$$

In the latter equation, for $y$, the Definition of Equation ({\ref{eq:v-smooth}}) implies 
$$\frac{\partial}{\partial y}\left.\left(\left(t-s\right)\cdot\left(\frac{y\left( s \right)}{2} -  \dot{x}\left( s \right)\right)\cdot y\left( s \right)\right)\right|_{s=t}=0,$$ $$\frac{\partial}{\partial y}\left.\frac{\partial}{\partial t}\left(\left(t-s\right)\cdot\left(\frac{y\left( s \right)}{2} -  \dot{x}\left( s \right)\right)\cdot y\left( s \right)\right)\right|_{s=t}=y\left( t \right)-\dot{x}\left( t \right),$$
$$\left.\frac{\partial}{\partial y}\frac{\partial^{k}}{\partial t^{k}}\left(\left(t-s\right)\cdot\left(\frac{y\left( s \right)}{2} -  \dot{x}\left( s \right)\right)\cdot y\left( s \right)\right)\right|_{s=t}=0,~for~k\geq2.$$ 

However, $y\left( t \right)-\dot{x}\left( t \right)$ is \textcolor{red}{zero} by Equation (\ref{eq:ex1}). That means the latter equation is $\infty$-smoothing rather than $2$-smoothing with respect to $y$.
\end{example}

 To find the correct signature matrix, we need to determine whether each of its elements is vanishing or not on the known constraints. Gr\"{o}bner bases \cite{Geddes1992} or Triangular Decomposition \cite{Collins1982} are possible approaches to this problem,
but they are high complexity and are only feasible in polynomial cases. Besides that, the signature matrix may be different for each component of constraints. In this section, we apply an efficient detection method \cite{Yang2021} to construct the corresponding signature matrix for each component of constraints.

\begin{prop}[Proposition $2.2.8$ of \cite{KrantzParks02}]\label{prop:real0}
If $f_1, ..., f_m$ are real analytic in some neighbourhood of
the point $\bm p\in \R^n$  and $g$ is real analytic in some neighbourhood of the point
$(f_1(\bm p),  ... , f_m(\bm p)) \in \R^m$, then the composition of functions $g(f_1(\bm x), ...,f_m(\bm x))$ is real analytic
in a neighborhood of $\bm p$.
\end{prop}

\begin{prop}[Proposition $2.2.3$ of \cite{KrantzParks02}]\label{prop:real1}
Let $f$ be a real analytic function defined  on an open subset $U \subset \R^n$, Then $f$ is continuous and has continuous, real analytic partial derivatives of all orders. Further, the indefinite integral with respect to any variable is real analytic.
\end{prop}

According to Proposition \ref{prop:real0} and Proposition \ref{prop:real1}, since  $\bm{\phi}$ and $\bm{\varphi}$ are real analytic, thus the {\IDAE} $\bm{F}$ is real analytic.

\begin{define}[Real zero set, singular set]\label{de:zero_set}
    The real zero set of a real analytic system $\bm f= \bm 0$ is denoted by $Z_{\R}(\bm f)$. The singular set of $\bm f= \bm 0$, denoted by $\Sing$, are those points at which $Z_{\R}(\bm f)$ is locally not an analytic manifold.
\end{define}
Then $Z_{\R}(\bm f) / \Sing =\cup_{i \in \mathcal{I}} C_i$
where each $C_i$ is a connected component of $Z_{\R}(\bm f) / \Sing$  of the analytic system. Moreover $C_i$ is an analytic manifold. 
\begin{define}[Component]\label{de:component}
We call $C_i$ a \textbf{component of} $Z_{\R}(\bm f)$. If $Z_{\R}(\bm f)$ is the zero set of the constraints of an analytic {\DAE}, then $C_i$ is called a \textbf{component of the constraints}.
\end{define}
 
 Consider a component $C$ of $Z_{\R}(\bm{F}^{(\bm{c})})$ with a real point $\bm p \in \R^n$.
Suppose $\rank \bm{\Jac}(\bm p) = r \leq n$. Without loss of generality, we assume that the sub-matrix $\bm{\Jac}(\bm p)[1:r,1:r]$ has full rank.

\begin{lemma}[Lemma $3.2$ of \cite{Yang2021}]\label{lem:measure0}
Let $Z_{\R}(\bm F)$ be the zero set of a real analytic system, $C$ be a component of $Z_{\R}(\bm F)$ in $\R^{m+n}$ of dimension $m$ and let $f$ be a real analytic function on $\R^{m+n}$.
Then the intersection $C\cap Z_{\R}(f)$ is equal to $C$ or has measure zero over $C$.
\end{lemma}

\begin{lemma}[Lemma $3.3$ of \cite{Yang2021}]\label{lem:whole}
Let $Z_{\R}(\bm F)$ be the zero set of a real analytic system, $C$ be a component of $Z_{\R}(\bm F)$. If $\bm{\Jac}[1:r,1:r]$ has full rank at some point $\bm p$ on $C$, then it is non-singular almost everywhere on $C$.
\end{lemma}

\begin{cor}\label{col:1}
   Let $Z_{\R}(\bm F)$ be the zero set of a real analytic system, $C$ be a component of $Z_{\R}(\bm F)$. If Equation ({\ref{eq:v-smooth}) or  Equation (\ref{eq:w-smooth}}) holds at an arbitrary point $\bm p$ on $C$, then it holds almost everywhere of the whole component. And, if Equation ({\ref{eq:v-smooth}) or Equation (\ref{eq:w-smooth}}) doesn't hold at this point, then it is $\infty$-smoothing or $\infty$-integral over the whole component.
\end{cor}
\begin{proof}
  According to Proposition \ref{prop:real0} and Proposition \ref{prop:real1}, since $\bm{\varphi}$ are real analytic, then  Equation ({\ref{eq:v-smooth}) and Equation (\ref{eq:w-smooth}}) are real analytic. If Equation ({\ref{eq:v-smooth}) or Equation (\ref{eq:w-smooth}}) holds, then the functions on the left hand sides of their equations are nonzero, otherwise these functions have measure zero. The proof of this corollary can be easily completed in a similar manner to the proof of Lemma \ref{lem:whole}.
  \foorp
\end{proof}

 Corollary \ref{col:1} yields Algorithm \ref{alg:1} for the detection method to construct a true signature matrix.
 
\begin{algorithm}
\caption{Construct Signature Matrix by Detection Method}\label{alg:1}
\begin{algorithmic}[1]
\Procedure{SigMat}{$\bm{F}$}
    \State Compute an initial point $\bm{p}$ from one component.
    \State Set $N:=10000$ \Comment{maximum number of times}.
    \For{$1 \leq i \leq n$} 
    \For{$1 \leq j \leq n$}
    \State Compute $[\sigma_{i,j}](\bm{\Phi})$ and $\LD(\varphi_i,x_j(s))$ by Equation (\ref{signature_DAE}).
    \For{$k$ from $1$ to $N$, $k:=k+1$}
        \State  $v_{i,j} := k$.
        \If{Equation (\ref{eq:v-smooth}) holds at the initial point $\bm{p}$}
        \State  Break.
        \EndIf
        \If {k == N}
         \State $v_{i,j}:=-\infty$, break.
        \EndIf
    \EndFor
    \EndFor
    \EndFor
    \State Compute $[\sigma_{i,j}](\bm{\Psi})$ by Equation (\ref{signature_IAE}).
     \State Compute $[\sigma_{i,j}](\bm{F})$ by Equation (\ref{signature}).
     \State Return $[\sigma_{i,j}](\bm{F})$.
\EndProcedure 
\end{algorithmic}
\end{algorithm}

If one or more solutions of the constraints can be found, the correct signature matrix on the component for each solution can be determined by the detection method. Actually, we can easily find a single solution to the constraints by root finding
methods, {\eg} Newton’s method {\cite{Dennis1996}}, SOR method {\cite{Ortega2000}}, tensor methods {\cite{Schnabel1984,Bader2005}}, {\etc}. Interval methods \cite{Pascal1997,Granvilliers2006} can find all points of the constraints if the scale and the interval is not too large. Homotopy methods \cite{Bates2013,Sommese2005} also can help to find at least one point from each component if the constraints are polynomials.

\section{Embedding Method for S-unamenable {\IDAE}s}\label{sec:implictit method}

There are some S-unamenable {\IDAE}s from applications, such as the {\IDAE} of a PID controller \cite{Visioli2003}. The $\Sigma$-method may also fail by producing a singular Jacobian, while these cases may be solvable. In this section, we give an improved structural method to regularize S-unamenable {\IDAE}s which can erase the difficulties mentioned in section \ref{s:sa}.

In a similar manner to \cite{Yang2021}, we divide such systems into two types: systems with \textbf{symbolic cancellation} (see Example \ref{ex:Zolf}) and systems with \textbf{numerical degeneration} (see Example \ref{ex:Degeneration}).

\begin{example}\label{ex:Degeneration} Numerical Degeneration:

%For the let-off and take-up system of a high-speed loom,  it is very important to ensure that the take-up and let-off do not cause warp deformation. This not only requires the coiling amount and let-off amount to be equal in the whole process, but also needs to ensure that the deformation energy is equal.\cite{Yang2013}  Then, we can describe the whole process as the following simplified {\IDAE} system.

Belt-drive systems and chain-drive systems, are important parts of mechanical transmission systems which are widely used in high-tech industries such as automobiles and high-speed railways \cite{Zhang2022}. In a Similar manner to let-off and take-up systems \cite{Yang2013}, they not only implicitly require the coiling amount and the let-off amount to be equal in the whole process, but also implicitly require that their energies are equal which helps to improve fatigue strength and to avoid heating caused by deformation. Their dynamic simulation models can be described as follow:

\[\left\{\begin{array}{rcc}
J_1\cdot \dot{\Omega}_{1}(t)+J_2\cdot\dot{\Omega}_{2}(t)+K\cdot\int_{t_0}^{t}{\left(\Omega_{1}(s)-\Omega_{2}(s)\right)ds}\\+B\cdot\left(\Omega_{1}(t)-\Omega_{2}(t)\right) -T_1(t)+T_2(t)&=&0\\
\int_{t_0}^{t}{\left(J_1\cdot(\Omega_{1}(s))^2-J_2\cdot(\Omega_{2}(s))^2 \right)ds}&=&0
\end{array}\right.\]
\[
\Rightarrow \bm{\Jac}=\left(\begin{array}{cc}
J_1&J_2\\
2\cdot J_1\cdot\Omega_{1}& -2\cdot J_2\cdot\Omega_{2}
\end{array} \right)\]
\begin{figure}[htpb]
  % Requires \usepackage{graphicx}
  \centering
  \includegraphics[height=3.5cm,width=8cm]{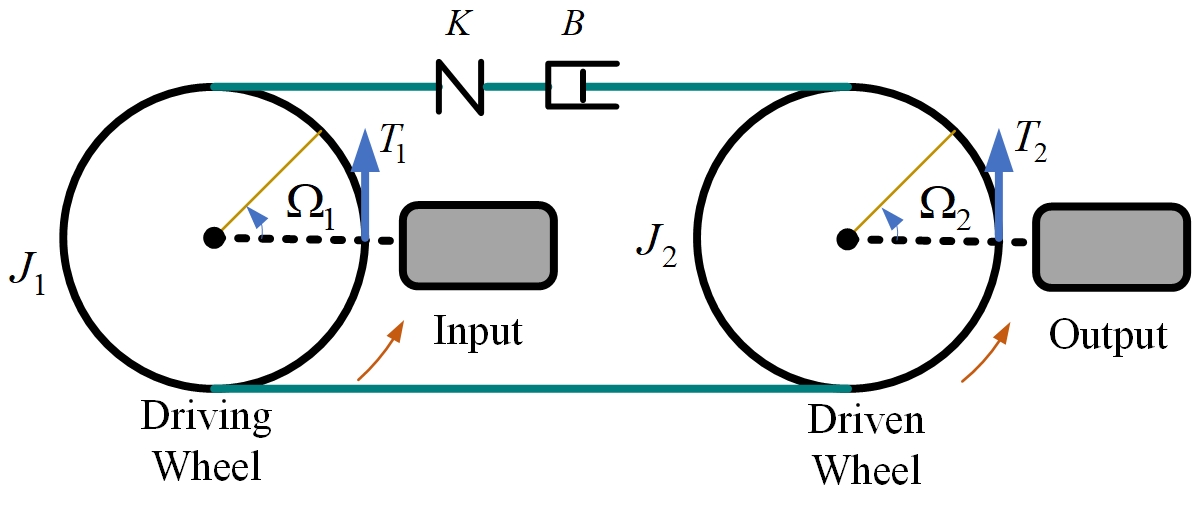}
   \caption{Belt-Drive System }
  \label{fig:loom}
\end{figure}

%Let $J_1=m_1\cdot r_1^2$ and $J_2=m_2\cdot r_2^2$ be moment of inertia, $r_1$ and $r_2$ be radius, respectively. They can be regard as constants in a short period of time. Suppose $m_1=m_2=m$ and $r_1=r_2=r$ at initial time, torques $T_1(t)$ and $T_2(t)$ are given functions, we try to  analyze the dynamic changes of the {\IDAE} system.
%Here, $J_1=m_1\cdot r_1^2$ and $J_2=m_2\cdot r_2^2$ be moment of inertia, $r_1$ and $r_2$ be radius, $K$ be given constant of the elastic coefficient, $\eta$ be parameter of transmission efficiency, $T_1(t)$ and $T_2(t)$ be given torques, respectively. When driving wheel and driven wheel are designed the same, then $m_1=m_2=m$ and $r_1=r_2=r$, we try to  analyze the dynamic changes of these systems.
Here, $J_1$ and $J_2$ are moments of inertia of the wheels, $K$ is a given constant of  elasticity coefficient, $B$ is a given constant damping coefficient, $T_1(t)$ and $T_2(t)$ are given torques, and $\Omega_1(t)$ and $\Omega_2(t)$ are angular velocities of the wheels, respectively. When the transmission ratio is equal, the moments of inertia of the driving wheel and the driven wheel are the same,
 that is $J_1=J_2=J$.

 In this example, the determinant of the Jacobian matrix is $-2\cdot J^2 \cdot (\Omega_1+\Omega_2)$. Since $J\cdot((\Omega_1)^2-(\Omega_2)^2)=(\Omega_1-\Omega_2)\cdot(\Omega_1+\Omega_2)\cdot J$ in the constraints, two consistent initial points can be selected from the two different components, respectively. If the point is on the component $\Omega_1-\Omega_2 = 0$, then
the $\Sigma$-method works well. But for any initial point on the component $\Omega_1+\Omega_2 = 0$, we always encounter a singular Jacobian, and we call this case numerical degeneration.

%In fact, almost all existing improved structural methods are modified the original {\IDAE}s which are very complex with integral. Neither symbolic cancellation  nor numerical degeneration of {\IDAE}s, the existing improved structural methods can not work well except the IRE method.

\end{example}

Especially, linear recombination or high multiplicity for an S-amenable {\IDAE} may also lead to the singularity of the Jacobian matrix. The singularity caused by linear recombination belongs to the case of symbolic
cancellation since the row vectors of the Jacobian matrix are linearly related. The singularity caused by high multiplicity belongs to the case of numerical degeneration since some constraints are ideals of the determinant of the Jacobian matrix.

%\subsection{S-unamenable {\IDAE}s}
%Similar to {\DAE}s, for S-unamenable cases of {\IDAE}s, structural analysis also may fail by producing a singular Jacobian, while these cases may be solvable. Symbolic cancellation and numerical degeneration are typical cases with singular Jacobian matrix.

Based on the definition of {\DOF} in section \ref{ss:dof}, our scheme for improved structural methods for an {\IDAE} can be described as: a scheme to construct a new {\IDAE} $\bm{G}$, whose solution of $\bm{x}$ is the same as {\IDAE} $\bm{F}$ and $0\leq DOF(\bm{G}) < DOF(\bm{F})$.

If we have an initial point from a component of an {\IDAE}, then according to Lemma \ref{lem:whole}, the rank of the Jacobian matrix
 is constant almost everywhere on this component. It can be calculated by singular value decomposition (SVD).  Moreover, the Jacobians with constant rank enable us to embed the zero set into a higher dimensional space. Hence, the embedding method is proposed in \cite{Yang2021} to construct a new {\DAE} $\bm{G}$ by decreasing the optimal value. But the embedding method for S-unamenable {\IDAE}s is invalid due to a negative optimal value which should be replaced by {\DOF}.

\subsection{Extension of the Embedding Method}\label{SS:exten}
%By the embedding method, an new {\IDAE} $\bm{G}$ can be constructed, whose solution space in $\bm{x}$ dimension is the same as {\IDAE} $\bm{F}$ and $0\leq DOF(\bm{G}) < DOF(\bm{F})$. In detail, the difference between $\bm{G}$ and $\bm{F}$ is their top block. In this section, we need to verify the feasibility of the embedding method for {\IDAE}s.

 Unlike other improved structural methods, the embedding method only replaces the variables of the top block.

\begin{cor}\label{col:2}
 The highest derivative of the top block $\bm{X}^{k_d}$ has no integral element.
\end{cor}
\begin{proof}
 The variables of the top block $\bm{X}^{k_d}= \{x_{1}^{(d_{1})},\cdots,x_n^{(d_n)}\}$. Since $\bm{d} = (d_1, \cdots, d_n) \geq \bm{0}$ is a constraint of the optimization problem ({\ref{LPP}}),  $\bm{X}^{k_d}$ is only related to $\bm{x}$ and its  derivatives.
    \foorp
\end{proof}

 Suppose $Z_{\R}(\bm{F}^{(\bm{c})})$ has constant rank i.e.
\begin{equation}\label{eq:rank}
\rank \bm{\Jac} = r =\rank \bm{\Jac}[1:r,1:r]  < n
\end{equation}
over a smooth component $C$ of $Z_{\R}(\bm{F}^{(\bm{c-1})})$.

%\textcolor{blue}{Wenqiang: I agree with you, Definition 5.1 is an Algorithm. I have modified it.}
%\begin{define}[The embedding Method]\label{define_IRE}
 %\sloppy{}
 %Suppose $(\bm{c},\bm{d})$ is the optimal solution of Problem (\ref{LPP}) for a given {\IDAE} $\bm{F}$. Then differentiated {\IDAE} $\bm{F}^{(\bm{c})} = \{  \bm{B}_{k_c},\bm{F}^{(\bm{c}-1)}\}$ has constant rank $\rank \bm{\Jac} = r < n$. Let $\bm{s}= (x_1^{(d_1)}, ..., x_r^{(d_r)})$, $\bm{y}= (x_{r+1}^{(d_{r+1})},...,x_n^{(d_n)})$ and $\bm{z}= (t,\bm{X},\bm{X}^{(1)},...,\bm{X}^{(k_d-1)})$. Then $\bm{B}_{k_c}=\{ \bm{f(s,y,z)},\bm{g(s,y,z)}\}$, where $\bm{f(s,y,z)}=\{F_1^{(c_1)},..., F_r^{(c_r)}\}$ with full rank Jacobian and $\bm{g(s,y,z)}=\{F_{r+1}^{(c_{r+1})},..., F_n^{(c_n)}\}$.
 %Then $\bm{G}$ is constructed by the following steps:

%\begin{enumerate}
%  \item  Introduce $n$ new equations $\hat{\bm{F}}=\{\bm{f(u,\xi,z)},\bm{g(u,\xi,z)}\}$:
%to replace $\bm{s}$ in the top block $\bm{B}_{k_c}$ by $r$ new dependent variables $\bm{u}=(u_1,...,u_r)$  respectively, and simultaneously replace $\bm{y}$ in the top block $\bm{B}_{k_c}$ by $n-r$ random constants $\bm{\xi}=(\xi_1,...,\xi_{n-r})$ respectively.
%  \item Construct a new square subsystem
%\begin{equation}\label{eq:newsys}
%\bm{F}^{aug} = \{ \bm{f(s,y,z)},\bm{\hat{F}(s,y,z)}\}.
%\end{equation}
%\item Construct $\bm{G}= \{\bm{F}^{aug} ,\bm{F}^{(\bm{c}-1)}\}$.
%\end{enumerate}
%\end{define}

%where $\bm{F}^{aug}$ has $n+r$ equations with $n+r$ leading variables $\{\bm{X}^{(k_d)}, \bm{u}\}$ and $\bm{X}^{(k_d)}=\{\bm{s},\bm{y}\}$.
Suppose $(\bm{c},\bm{d})$ is the optimal solution of Problem (\ref{LPP}) for a given {\IDAE} $\bm{F}$. Then differentiated {\IDAE} $\bm{F}^{(\bm{c})} = \{  \bm{B}_{k_c},\bm{F}^{(\bm{c}-1)}\}$ has constant rank $\rank \bm{\Jac} = r < n$. Let $\bm{s}= (x_1^{(d_1)}, ..., x_r^{(d_r)})$, $\bm{y}= (x_{r+1}^{(d_{r+1})},...,x_n^{(d_n)})$ and $\bm{z}= (t,\bm{X},\bm{X}^{(1)},...,\bm{X}^{(k_d-1)})$. Then $\bm{B}_{k_c}=\{ \bm{f(s,y,z)},\bm{g(s,y,z)}\}$, where $\bm{f(s,y,z)}=\{F_1^{(c_1)},..., F_r^{(c_r)}\}$ with full rank Jacobian and $\bm{g(s,y,z)}=\{F_{r+1}^{(c_{r+1})},..., F_n^{(c_n)}\}$. Then $\bm{G}$ is constructed by the embedding method in Algorithm \ref{alg:2}.

\begin{algorithm}
\caption{The Embedding Method}\label{alg:2}
\begin{algorithmic}[1]
\Procedure{Embedding}{}    
    \State Introduce $n$ new equations $\bm{\hat{F}(s,y,z)}:=\bm{B}_{k_c}$.
    \State to replace $\bm{s}$ in the top block $\hat{\bm{F}}$ by $r$ new dependent variables $\bm{u}=(u_1,...,u_r)$  respectively.
    \State to replace $\bm{y}$ in the top block $\hat{\bm{F}}$ by $n-r$ random constants $\bm{\xi}=(\xi_1,...,\xi_{n-r})$ respectively.
   \State Construct a new square subsystem
        \begin{equation}\label{eq:newsys}
        \bm{F}^{aug} := \{ \bm{f(s,y,z)},\bm{\hat{F}(u,\xi,z)}\}.
        \end{equation}
        \Comment{Where $\bm{F}^{aug}$ has $n+r$ equations with $n+r$ leading variables $\{\bm{X}^{(k_d)}, \bm{u}\}$ and $\bm{X}^{(k_d)}=\{\bm{s},\bm{y}\}.$
    \State Construct $\bm{G} := \{\bm{F}^{aug} ,\bm{F}^{(\bm{c}-1)}\}$}.
    \State Return $\bm{G}$.
\EndProcedure
\end{algorithmic}
\end{algorithm}

Most preferably, the initial values of the new variables $\bm{u}$ can simply be taken as the initial values of their replaced variables $\bm{s}$. And $\bm{\xi}$ takes the same initial value as was assigned to $\bm{y}$.

\begin{theorem}\label{thm:result}
Let $(\bm{c},\bm{d})$ be the optimal solution of Problem (\ref{LPP}) for a given {\IDAE} $\bm{F}$.
Let $\bm{F}^{(\bm{c})} = \{ \bm{B}_{k_c}, \bm{F}^{(\bm{c}-1)} \}$ as defined in Equation (\ref{eq:B_i}). If $\bm{F}^{(\bm{c})}$ satisfies (\ref{eq:rank}), and $C$ is a component of constraints, then
$$Z_{\R}(\bm{F}^{(\bm{c})})\cap C = \pi Z_{\R}(\bm{G}) \cap C$$
where $\bm{G}= \{\bm{F}^{aug},\bm{F}^{(\bm{c}-1)} \}$ as defined in Algorithm \ref{alg:2}.   %when the Jacobian matrix of the top block of prolongation
%of $\bm{G}$ is nonsingular,
 Moreover, we have $DOF(\bm{G}) \leq  DOF(\bm{F}) - (n-r)$.
\end{theorem}
\begin{proof}
Just like the proof of Theorem $4.3$ in \cite{Yang2021}, since
$\bm{F}^{(\bm{c}-1)}$ is common to
both ${\bm{F}^{(\bm{c})}}$ and ${\bm{G}}$, we have $Z_{\R}(\bm{F}^{(\bm{c})})\cap C = \pi Z_{\R}(\bm{G}) \cap C$.

We construct a pair $(\hat{\bm{c}},\hat{\bm{d})}$, for $i=1, \cdots, n$, $\hat{c}_{i}=0$ and for $ j=1, \cdots, n$, $\hat{d}_{j}= d_{j}$. Since $(\bm{c},\bm{d})$ is the optimal solution for $\bm{F}$, and $\bm{B}_{k_c}$ is the top block of $\bm{F}^{(\bm{c})}$, it follows that $(\hat{\bm{c}},\hat{\bm{d})}$ is the optimal solution for $\bm{B}_{k_c}$, and $\delta(\bm{B}_{k_c})= \sum\limits_{j}{d_j}$.

In the same manner of \cite{Yang2021}, we also construct a pair of feasible solutions $(\bar{\bm{c}},\bar{\bm{d})}$ for $\bm{F}^{aug}$, which can help us to obtain
$\delta(\bm{F}^{aug})\leq \delta(\bm{B}_{k_c}) -(n-r)$. Where
\begin{equation}\label{eq:feasible_sln}
 \begin{array}{ll}
 \bar{c}_{i}= \left\{%
\begin{array}{ll}
    0,& i=1,\cdots,r; \\
    1,& i=(r+1),\cdots,(n+r). \\
\end{array}%
\right.
\\
\bar{d}_{j}= \left\{%
\begin{array}{ll}
    d_{j},& j=1,\cdots,n; \\
    1,& j=(n+1),\cdots,(n+r). \\
\end{array}%
\right.
\end{array}
\end{equation}

According to Corollary \ref{col:2}, the replaced variables only occur in $\bm{X}^{k_d}$ in the top block, thus $\bm{F}^{aug}$ and $\bm{B}_{k_c}$ have the same integration variables. By Proposition \ref{prop:DOF}, such that $DOF(\bm{F}^{aug}) \leq DOF(\bm{B}_{k_c}) -(n-r)$.

Obviously, since both ${\bm{F}^{(\bm{c})}}$ and ${\bm{G}}$ have the same
block of constraints $\bm{F}^{(\bm{c}-1)}$, according to Proposition {\ref{define_delta1}}, it follows that $DOF({\bm{G}})-DOF({\bm{F}^{(\bm{c})}})=DOF(\bm{F}^{aug})-DOF(\bm{B}_{k_c})\leq -(n-r)$. Finally,  $DOF(\bm{G}) \leq DOF(\bm{F}^{(\bm{c})}) - (n-r)= DOF(\bm{F}) - (n-r)$, since $DOF(\bm{F})=DOF(\bm{F}^{(\bm{c})})$ by Proposition \ref{prop:DOF}. \foorp \\
\end{proof}

Moreover, since the embedding method only replaces derivatives for $\bm{d}\geq \bm{0}$,  the $\omega_{i,j}$-integral of variables in $\bm F$ is the same as the $\omega_{i,j}$-integral of variables in $\bm G$.

Although there are more dependent variables in $\bm{G}$,
the computational cost is much lower than explicit symbolic elimination, since $\bm{G}$ and the corresponding lifted initial points can be easily constructed.
%Moreover, in the embedding method, the feasible solution $(\bar{\bm{c}},\bar{\bm{d})}$ given in Equation (\ref{eq:feasible_sln}) without {\ILP} solving is an optimal solution in almost every example.
Theoretically, Lemma {\ref{lem:lifting}} below
shows that the feasible solution $(\bar{\bm{c}},\bar{\bm{d})}$ is optimal under some reasonable assumptions.

\begin{lemma}\label{lem:lifting}
Suppose each equation $F_{i}$ in the top block $\bm{B}_{k_c}$ of a {\IDAE} $\bm F$
contains at least one variable  $x_{j}\in \bm{X}^{(k_d)-1}$.  If $\bm F$ is also a perfect match, then $(\bar{\bm{c}},\bar{\bm{d})}$ in Equation (\ref{eq:feasible_sln})
is an optimal solution, and $DOF(\bm{G}) = DOF(\bm{F}) - (n-r)$.
\end{lemma}

 This lemma is proved by contradiction. For more detail please see the proof of Lemma $4.4$ in \cite{Yang2021}.

\subsection{Examples}\label{ssec:exam}
\begin{example}{(Symbolic Cancellation)}\label{Zolf_3}
    According to Example {\ref{ex:Zolf}} and Example {\ref{ex:Zolf_2}}, this {\IDAE} is a typical example of symbolic cancellation.

Obviously, we still cannot solve the system directly after the $\Sigma$ method. Fortunately, as shown in \cite{ZOLF2021}, the ES-method can successfully regularize it, while the LC-method fails.

Here, we apply the embedding method to this example. According to Algorithm \ref{alg:2}, we have $\bm{s}=\{{x}_{1}\}$,$\bm{y}=\{{x}_{2}\}$, $\bm{f(s,y,z)}=\{F^{(1)}_{2}\}$, $\bm{g(s,y,z)}=\{F_1\}$, $\bm{F}^{(\bm{c}-1)}=\{F_2\}$.
Thus, $\hat{\bm{F}}=\{\bm{f(u,\xi,z)},\bm{g(u,\xi,z)}\}$,
where $\bm{s}$ and $\bm{y}$ are replaced by $u$ and some random constants $\xi$ respectively. Thus,

\[{\bm{F}^{aug}}=\left\{\begin{array}{l}
x_1(t)+x_2(t)+\int_{t_0}^{t}{ x_1(s)\cdot x_2(s)ds} - \dot{g}_2(t)\\
e^{-u(t)-\xi}-g_1(t)\\
u(t)+\xi+\int_{t_0}^{t}{ x_1(s)\cdot x_2(s)ds} - \dot{g}_2(t)
\end{array}\right.\]

After executing the embedding method, directly construct  $\bar{\bm{c}}=(0,1,1)$ and $\bar{\bm{d}}=(0,0,1)$ by Lemma \ref{lem:lifting}. Actually, it also is the optimal solution of {\AP} by calculation. And the {\DOF} of the new system $\bm{G}$ is $\sum{\bar{d}_{j}}-\sum{\bar{c}_{i}}+\sum\limits_{j}{\max\limits_{i}{\omega_{i,j}}}-\#eqns(\bm{F}^{(\bm{c}-1)})=1-2+4-1 =DOF(\bm{F})-n+r = 3-2+1$, which is the same as the {\DOF} after the ES-method.

Then we can verify that the determinant of the new Jacobian matrix is $(x_2-x_1)\cdot e^{u-\xi}$, which is non-singular at $t_0$ if $x_1(t_0)-x_2(t_0) \neq 0$, which is the same result as in \cite{ZOLF2021}.
\end{example}

\begin{example}{(Numerical Degeneration)}\label{ex:2}
Consider the following {\IDAE} with dependent variables $x \left( t \right)$ and $y \left( t \right)$:
\[
\bm{F} = \left\{2\,y {\frac {{\rm d^{2}}x}{{\rm d}t^{2}}}  - x
  {\frac {{\rm d^{2}}y }{{\rm d}t^{2}}} +2x
 \left({\frac {{\rm d}x}{{\rm d}t}} \right)^{2} - {\frac {{\rm d} x }{{\rm d}t}} +\sin \left( t \right) , \int_{0}^{t} {\left(y\left( s \right) -  x\left( s \right)^2\right)}ds \right\}.
\]

 The exact solution of this {\IDAE} is $x(t)=C-\cos(t)$ and $y(t)=x(t)^2$. Applying the structural method yields $\bm{c}=(0,3)$ and $\bm{d}=(2,2)$. Then $\bm{F}^{(\bm{c})}= [\{2y x_{tt}-x y_{tt} +2x {x_t}^{2}-x_t+\sin(t), y_{tt}-2x_t^{2}-2xx_{tt}\}, \{-2x x_t+y_t\}, \{-x^2+y\},\{\int{\left(-x^2+y\right)}ds\}]$, and the Jacobian matrix of the top block is
  $\bm{\Jac} =
      \left(
        \begin{array}{cc}
          2y & -x \\
          -2x & 1 \\
        \end{array}
      \right)$.

 Although the determinant of the Jacobian $2y-2x^2$ is not identically zero, it must equal zero at any initial point since
the determinant belongs to the ideal generated by the hidden constraints. Thus,

\[\bm{F}^{aug}=\left\{\begin{array}{rcc}
2\,y {\frac {{\rm d^{2}}x}{{\rm d}t^{2}}}  - x
  {\frac {{\rm d^{2}}y }{{\rm d}t^{2}}} +2x
 \left({\frac {{\rm d}x}{{\rm d}t}} \right)^{2} - {\frac {{\rm d} x }{{\rm d}t}} +\sin \left( t \right)&=&0\\
2\cdot u_{1}\cdot y-\xi\cdot x+2x
\left({\frac {{\rm d}x}{{\rm d}t}} \right)^{2} - {\frac {{\rm d}x}{{\rm d}t}} +\sin(t)&=&0\\
\xi-2\cdot  u_{1}\cdot  x-2\cdot\left({\frac {{\rm d}x}{{\rm d}t}} \right)^{2}&=&0
\end{array}\right.\]

After the embedding method with $\bm{s}=\{{\frac {{\rm d^{2}}x}{{\rm d}t^{2}}}\}$, $\bm{y}=\{{\frac {{\rm d^{2}}y}{{\rm d}t^{2}}}\}$, $\bm{f(s,y,z)}=\{F_{1}\}$, $\bm{g(s,y,z)}=\{F^{(3)}_2\}$, $\bar{\bm{c}}=(0,1,1)$ and $\bar{\bm{d}}=(2,2,1)$, the new Jacobian matrix of $\bm{F}^{aug}$ is
$$\bm{\Jac} =
      \left(
        \begin{array}{ccc}
         2y  &-1&0\\
          4x\cdot x_t -1 & 0 & 2y \\
          -4x_t & 0& -2x \\
        \end{array}
      \right)$$

It is obvious that the determinant of the new Jacobian matrix will not degenerate to a singular matrix by virtue of the constraints.

It should be noted that there is a redundant constraint in this example which will affect the {\DOF}.

\end{example}

%As pointed out in Example \ref{ex:2}, numerical degeneration can be defined as that $\det \bm{\Jac}_{k_c}$ may not be identical zero, but $\det \bm{\Jac}_{k_c}=0$ at any consistent initial point of $Z(\bm{F}^{(\bm{c})})$——-the zero set of $\bm{F}^{(\bm{c})}$. Since $\bm{F}$ is a polynomial system in $\{\bm{x},\bm{x}^{(1)},...,\bm{x}^{(\ell)} \}$,
%$\bm{F}^{(\bm{c})}$ can be considered as a polynomial system in the variables $\{\bm{X}^{(0)},...,\bm{X}^{(k_d)}\}$. In the language of algebraic geometry,
%it means that $\det \bm{\Jac}_{k_c} \in \sqrt[\R]{\langle \bm{F}^{(\bm{c})} \rangle }$ or equivalently $Z_{\R}(\bm{F}^{(\bm{c})}) \subseteq Z_{\R}(\bm{\Jac}_{k_c})$.

%The above two examples show that the embedding method can be successful only after one step of regularization. However, in some case, it needs more than once as shown in Example \ref{ex:pendulum1}.

\begin{example}{(Hybrid System)}\label{ex:3}
Consider an S-amenable {\IDAE} with exact solution $x_1 \left( t \right) = 1 - 2t$ and $x_2\left( t \right)=4t$:
\[\bm{F} = \left\{\begin{array}{ll} F_1: & t\,x_1\left( t \right)+ \int_{0}^{t} { x_2\left( s \right)} ds - t\\
F_2: & \int_{0}^{t}{x_1\left( s \right)} ds + t^{2} - t
\end{array}
\right.
\]

Whose signature matrix is
$$[\sigma_{i,j}](\bm{F})=\begin{array}{ccc}
 & & c_i\\
 & \left(\begin{array}{cc}
   0 & -1\\
    \color{blue}\bm{-1} & \color{blue}\bm{-\infty}
 \end{array}\right) & \begin{array}{c}
   1\\
   \color{blue}\bm{2}
 \end{array}\\ d_j&\begin{array}{cc}
   1 & 0
 \end{array} &
\end{array} $$
and {\DOF} is $0$.

When linear recombination and high multiplicity occur in it such as $\bm{\bar{F}} =\left\{ F_1^{2}, F_1 - F_2\right\} $, it turns out to be an S-unamenable {\IDAE}.

By the $\Sigma$-method, the signature matrix

$$[\sigma_{i,j}](\bm{\bar{F}})= \begin{array}{ccc}
 & & c_i\\
 & \left(\begin{array}{cc}
   0 & -1\\
    \color{blue}\bm{0} & \color{blue}\bm{-1}
 \end{array}\right) & \begin{array}{c}
   1\\
   \color{blue}\bm{1}
 \end{array}\\ d_j&\begin{array}{cc}
   1 & 0
 \end{array} &
\end{array}.$$
Here, {\DOF} is $2$  with $1$ redundant constraint, and Jacobian matrix is $\bm{\Jac} =
      \left(
        \begin{array}{cc}
         2\cdot t\cdot F_1  & 2\cdot F_1\\
         t & 1
        \end{array}
      \right)$ is identical zero, whose rank is $1$.

In this example, we need to call the embedding method twice to find $2$ hidden constraints. At the first call, ${s}=\{{\frac {{\rm d}x_1}{{\rm d}t}}\}$, ${y}=\{x_2\}$, ${f(s,y,z)}=\{\bar{F}'_{2}\}$, ${g(s,y,z)}=\{\bar{F}'_{1}\}$, $\bar{\bm{c}}=(0,1,1)$ and $\bar{\bm{d}}=(1,0,1)$. The Jacobian matrix of the new subsystem $\bm{\bar{F}}^{aug}$ is
$\bm{\Jac} =
      \left(
        \begin{array}{ccc}
         t  & 1 & 0\\
          2\cdot t\cdot (\bar{F}'_1(u,\xi,z)+\bar{F}_1)& 2\cdot\bar{F}'_1(u,\xi,z) & 2\cdot t\cdot \bar{F}_1 \\
          0 & 0 & t \\
        \end{array}
      \right)$, which implies that the new system has numerical degeneration.
      At the second call of the embedding method,  $\bm{s}=\{{\frac {{\rm d}x_1}{{\rm d}t}},{\frac {{\rm d}u}{{\rm d}t}}\}$, ${y}=\{x_2\}$, $\bm{f(s,y,z)}=\{{\bar{F}}^{aug}_{1},{\bar{F}}^{aug}_{3}\}$, and ${g(s,y,z)}=\{{\bar{F}}^{aug}_{2}\}$. Finally, the Jacobian matrix is non-singular.
\end{example}

The above examples illustrate that the embedding method is a good choice for S-unamenable {\IDAE}s including those with linear recombination and high multiplicity cases.

\section{Global Numerical Solution of Two Stage Drive System}\label{s:global}

%\section{Examples}
After structural analysis, a low-index {\IDAE} can be obtained which can be decoupled into a system of regular Volterra integro differential equations ({\VIDE})s and a system of second-kind Volterra integral equations ({\VIE})s \cite{Liang2019}. Generally, numerical solution methods for {\IDAE}s can be summarized in terms of two steps. The first step is to compute an initial value by {\VIE}s, and the second step is to solve a {\VIDE} using the initial value of the first step and to check whether the new solution conforms to {\VIE}s. Most studies have focused on the second step. Some numerical iterative formats for some typical {\IDAE} systems have been proposed. Implicit Runge-Kutta methods \cite{KAUTHEN1993}, collocation methods and collocation based methods \cite{Liang2019,brunner_2004}, implicit Euler methods and methods based on backward differentiation formulas \cite{Bulatov11,Bulatov06}. 

%For the initial value, a guess method can be used to locally select a point on an uncertain component.

%For polynomial {\IDAE}s, real witness points of {\VIE}s at intial time (there is no integral item in {\VIE}s at this time) can be calculated by the Homotopy continuation method \cite{Yang2021}, which can help us to find initial points form all components. Thus, we give a frame diagram for global numerical solution of a low index {\IDAE}, as shown in Figure \ref{fig:frame}. Next, we will give an example to illustrate it.

 The global numerical method in \cite{Yang2021} can be applied to {\IDAE}s directly. In particular, Homotopy continuation methods or interval methods can help to find all initial points. Next, we will give an example to illustrate the global numerical method in \cite{Yang2021}.  %We give a frame diagram for global numerical solution of a low index {\IDAE}, as shown in Figure \ref{fig:frame}. Next, we will give an example to illustrate it.

%\begin{figure}[htpb]
  % Requires \usepackage{graphicx}
% \centering
 % \includegraphics[height=5cm,width=12cm]{shceme.jpg}
 %  \caption{Global Numerical Solution for {\IDAE}s}
  %\label{fig:frame}
%\end{figure}

%\subsection{Non-linearly Modified Pendulum}\label{ssec:exam2}

%\subsection{The Let-off and Take-up System}\label{ssec:exam4}
\begin{figure}[htpb]
  % Requires \usepackage{graphicx}
  \centering
  \includegraphics[height=7cm,width=12cm]{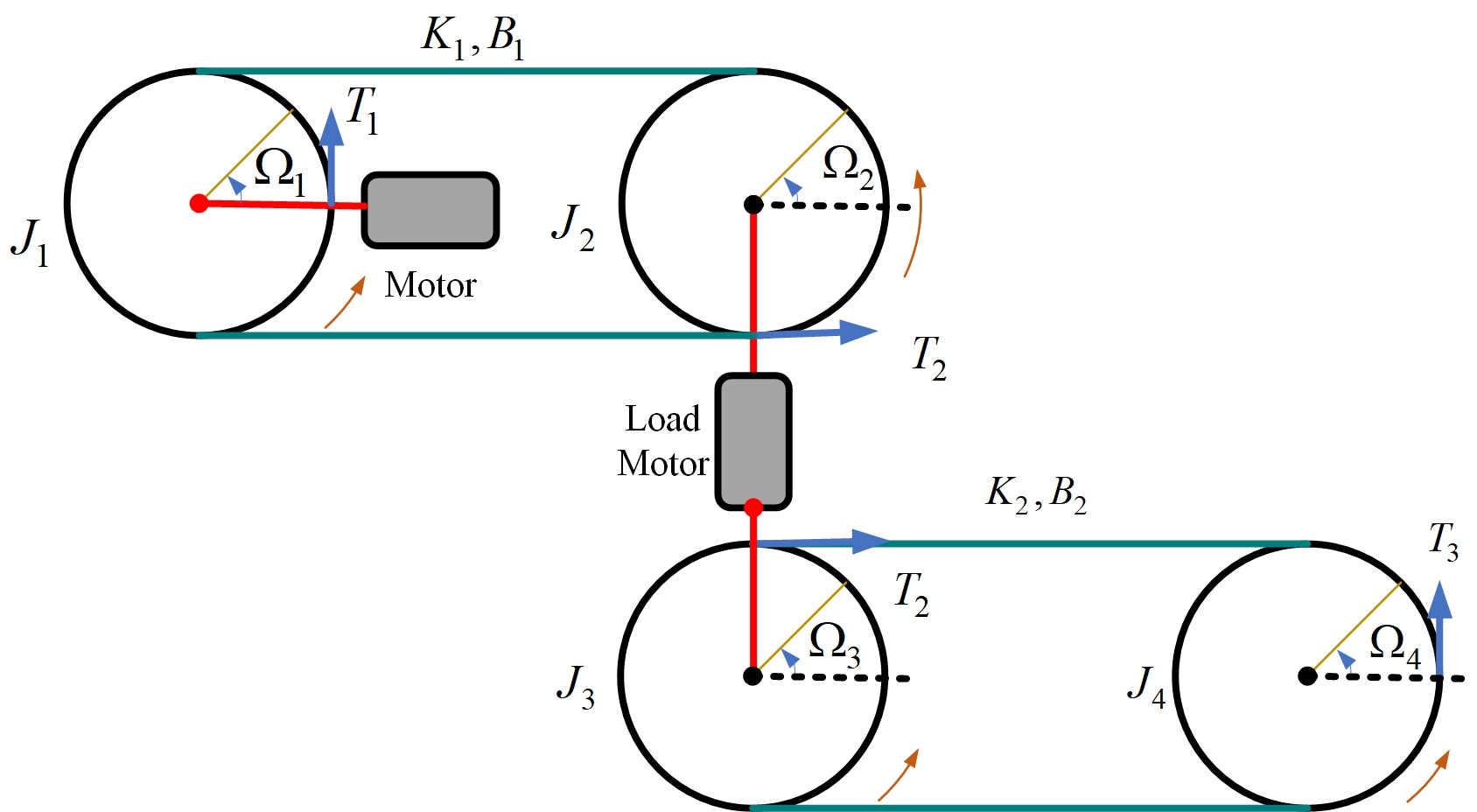}
   \caption{The Two Stage Drive System}
  \label{fig:trans}
\end{figure}

The specific description of one stage driven system is given in Example \ref{ex:Degeneration}. In applications, we can usually introduce a constant load in a series of one stage driven systems to achieve multi-stage transmission. When it comes to two stage drive system in Figure \ref{fig:trans}, it can be described as follows.

Assume moments of inertia $J_1=J_2=J_3=J_4=1$, elastic
coefficients $K_1=K_2=1$, damping coefficients $B_1=B_2=1$, and torques $T_1(t)=2$, $T_2(t)=-sin(t)$, $T_3(t)=1$, respectively.

\[\left\{\begin{array}{rcc}
\dot{\Omega}_{1}+\dot{\Omega}_{2}+\int_{0}^{t}{\left(\Omega_{1}-\Omega_{2}\right)ds}+\Omega_{1}-\Omega_{2}+2-sin(t)&=&0\\
\int_{0}^{t}{\left((\Omega_{1})^2-(\Omega_{2})^2 \right)ds}&=&0\\
\dot{\Omega}_{3}+\dot{\Omega}_{4}+\int_{0}^{t}{\left(\Omega_{3}-\Omega_{4}\right)ds}+\Omega_{3}(t)-\Omega_{4}(t)+sin(t)-1&=&0\\
\int_{0}^{t}{\left((\Omega_{3})^2-((\Omega_{4})^2 \right)ds}&=&0
\end{array}\right.\]

%\[\left\{\begin{array}{rcc}
%J_1\cdot \ddot{\theta}_{1}(t)+J_2\cdot\ddot{\theta}_{2}(t)+K_1\cdot\int_{0}^{t}{\left(\dot{\theta}_{1}(s)-\dot{\theta}_{2}(s)\right)ds}+T_1(t)+T_2(t)&=&0\\
%\int_{0}^{t}{\left(J_1\cdot(\dot{\theta}_{1}(s))^2-(J_2\cdot(\dot{\theta}_{2}(s))^2 \right)ds}&=&0\\
%J_3\cdot \ddot{\theta}_{3}(t)+J_4\cdot\ddot{\theta}_{4}(t)+K_2\cdot\int_{0}^{t}{\left(\dot{\theta}_{3}(s)-\dot{\theta}_{4}(s)\right)ds}-T_2(t)+T_3(t)&=&0\\
%\int_{0}^{t}{\left(J_3\cdot(\dot{\theta}_{3}(s))^2-(J_4\cdot(\dot{\theta}_{4}(s))^2 \right)ds}&=&0
%\end{array}\right.\]

%\[\Rightarrow \bm{\Jac}=\left(\begin{array}{cccc}
%J_1&J_2&0&0\\
%2\cdot J_1\cdot\dot{\theta}_{1}& -2\cdot %J_2\cdot\dot{\theta}_{2}&0&0\\
%0&0&J_3&J_4\\
%0&0&2\cdot J_3\cdot\dot{\theta}_{3}& -2\cdot %J_4\cdot\dot{\theta}_{4}\\
%\end{array} \right)\]

\[\Rightarrow \bm{\Jac}=\left(\begin{array}{cccc}
1&1&0&0\\
2\cdot\Omega_{1}& -2\cdot\Omega_{2}&0&0\\
0&0&1&1\\
0&0&2\cdot\Omega_{3}& -2\cdot \Omega_{4}\\
\end{array} \right)\]

Here, the two stage driven system is designed to be an equal transmission ratio system. It must be a numerically degenerate system with $4$ components in Table \ref{tab:COMP}.

 By structural analysis, the dual optimal solutions is $\bm{c}=(0,2,0,2)$ and $\bm{d}=(1,1,1,1)$. In this example, there are two separate equation blocks that we can deal with them by applying the embedding method separately to reduce complexity. We also can construct the optimal solutions of the new system $\bm{G}$ by Lemma \ref{lem:lifting}.

\begin{table}[htpb]
	\caption{Components of Two Stage Driven System  }\label{tab:COMP}
\centering
\begin{tabular}{|c|c|c|c|c|c|}
  \hline
  % after \\: \hline or \cline{col1-col2} \cline{col3-col4} ...
  &Components & $\rank \bm{\Jac}$ &$\bm{f(s,y,z)}$ & $\bm{s}$& Method\\
  \hline
  (a) & $\Omega_{1}=\Omega_{2}$, $\Omega_{3}=\Omega_{4}$ & $4$ &&& $\Sigma$\\
   \hline
 (b) & $\Omega_{1}=-\Omega_{2}$, $\Omega_{3}=\Omega_{4}$ & $3$& $F_2$&$\dot{\Omega}_{1}$& embedding\\
   \hline
 (c) & $\Omega_{1}=\Omega_{2}$, $\Omega_{3}=-\Omega_{4}$ & $3$&$F_4$&$\dot{\Omega}_{3}$& embedding\\
 \hline
 (d) & $\Omega_{1}=-\Omega_{2}$, $\Omega_{3}=-\Omega_{4}$ & $2$&$F_2$,$F_4$&$\dot{\Omega}_{1}$,$\dot{\Omega}_{3}$& embedding\\
  \hline
\end{tabular}
\end{table}

When $t\in [0,5]$, four witness points from each component
 are computed by the Homotopy continuation method \cite{WWX2017} where each point has
 coordinates $(\Omega_{1}, \Omega_{2}, \Omega_{3}, \Omega_{4})$:
 \[
\begin{array}{rrrr}
(0.21862079540 & 0.21862079540 & -0.87716795773 & -0.87716795773 )\\
(-1.0000000000 & 1.0000000000 &  -0.87716795773 & -0.87716795773 )\\
(0.21862079540 & 0.21862079540 & 0.50000000000 & -0.50000000000) \\
(-1.0000000000 & 1.0000000000 & 0.50000000000 & -0.50000000000)\\
\end{array}
\]

These witness points are approximate points near the consistent initial value points, which need to be refined by Newton iteration. Finally, four numerical solutions from different components are shown in Figure \ref{fig:2sd}.

\begin{figure}[htbp]
\centering
\subfigure[]{
\includegraphics[width=6cm,height=4.5cm]{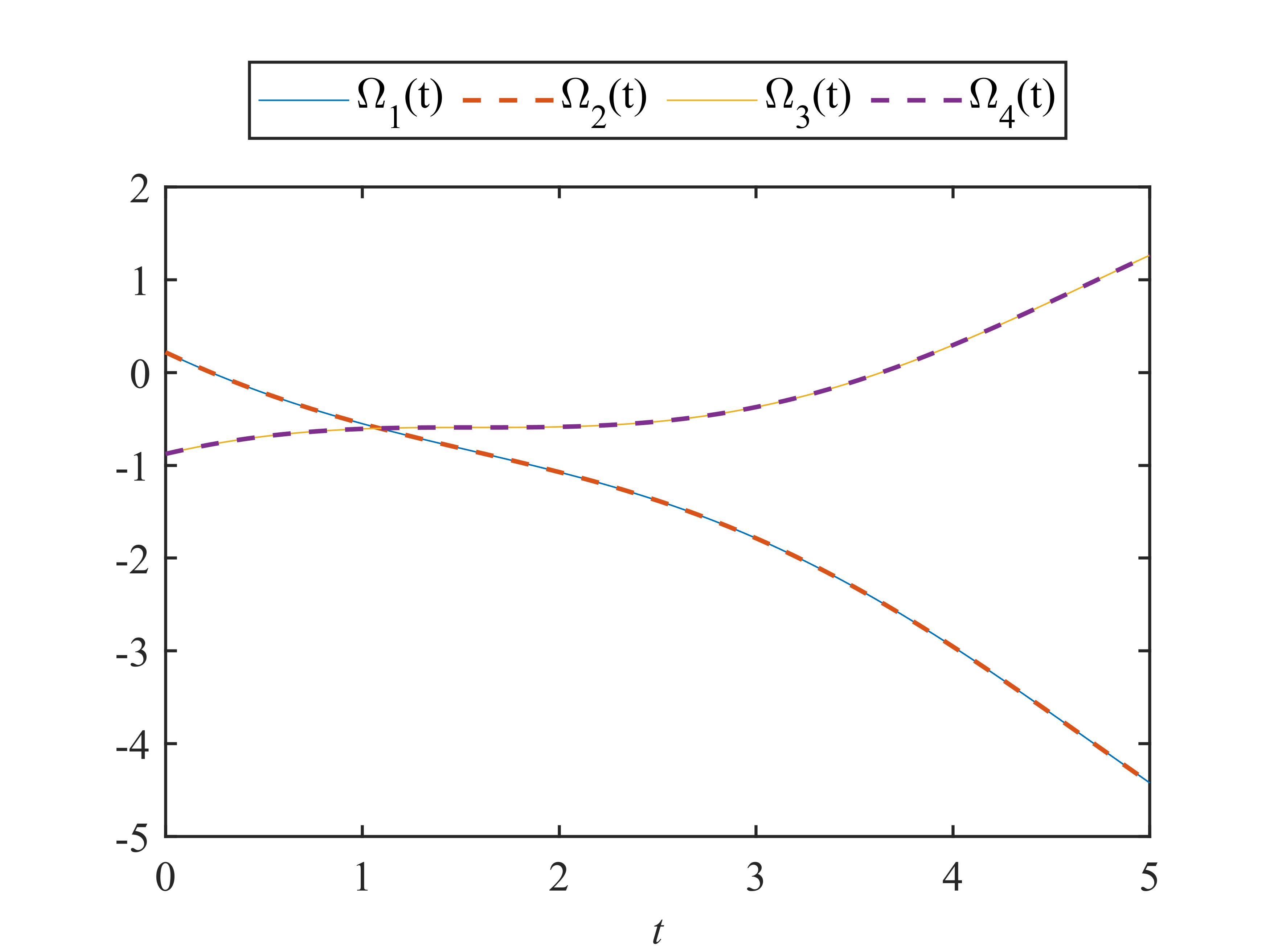}
%\caption{fig1}
}
\quad
\subfigure[]{
\includegraphics[width=6cm,height=4.5cm]{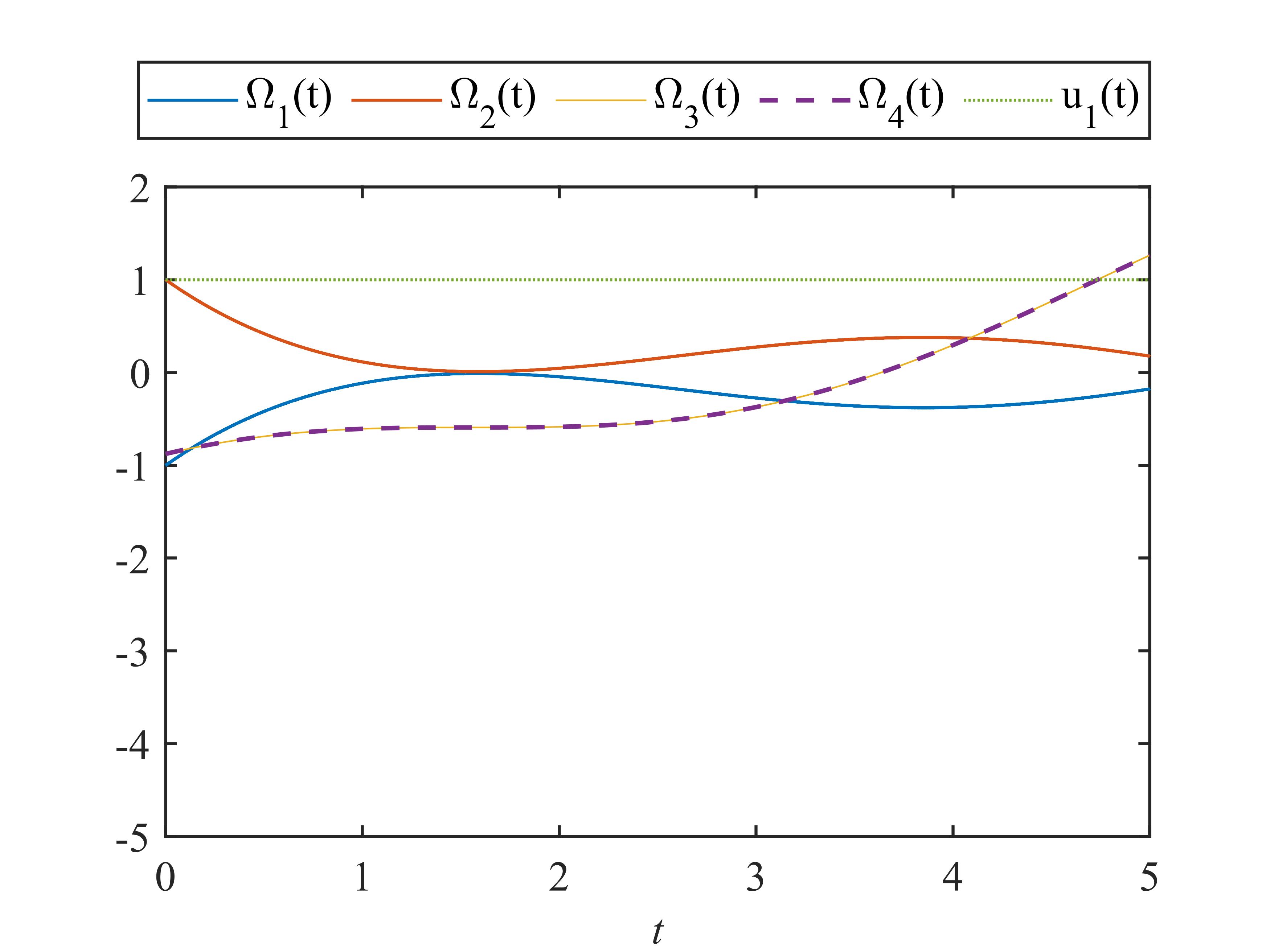}
}
\quad
\subfigure[]{
\includegraphics[width=6cm,height=4.5cm]{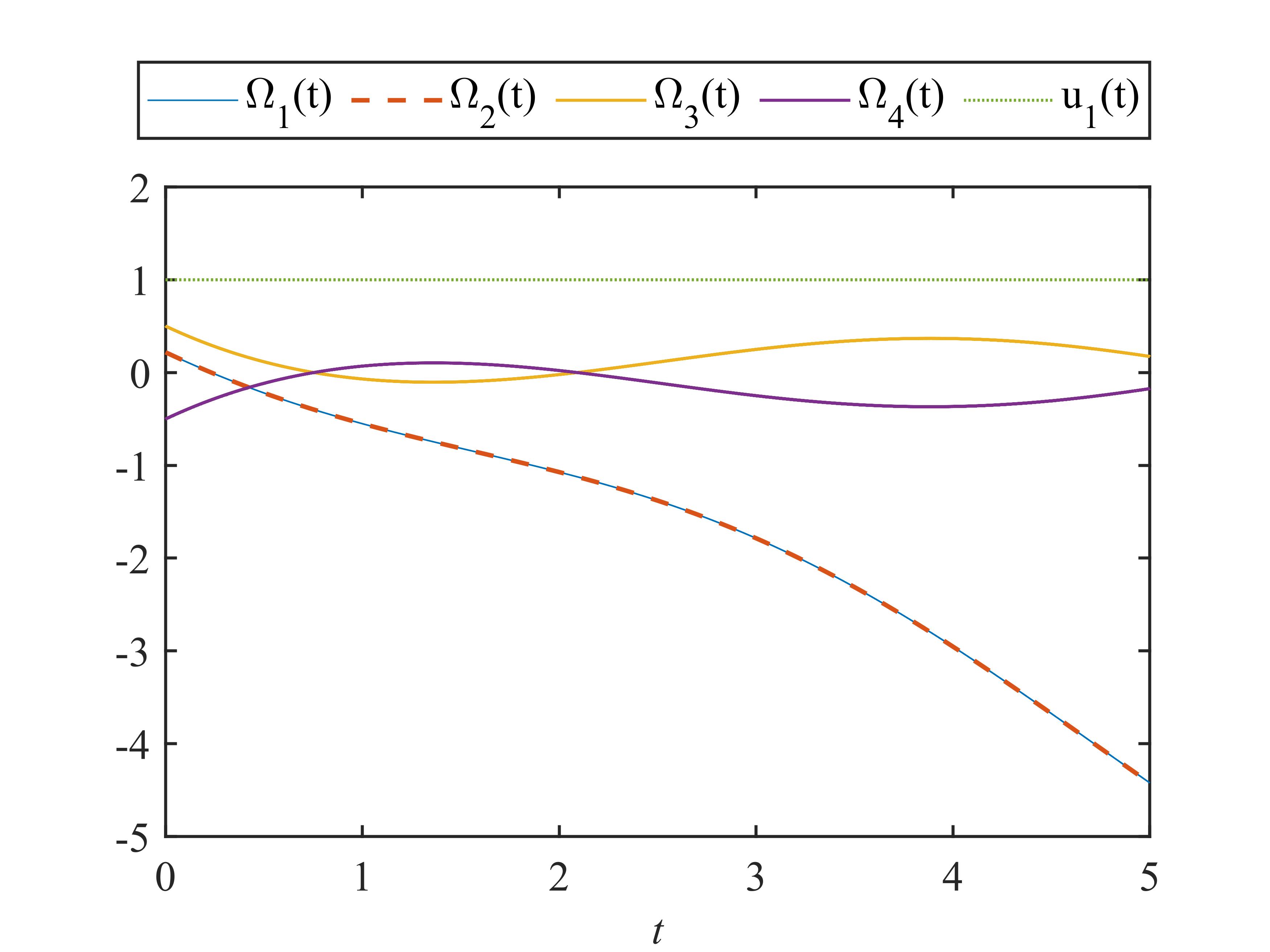}
}
\quad
\subfigure[]{
\includegraphics[width=6cm,height=4.5cm]{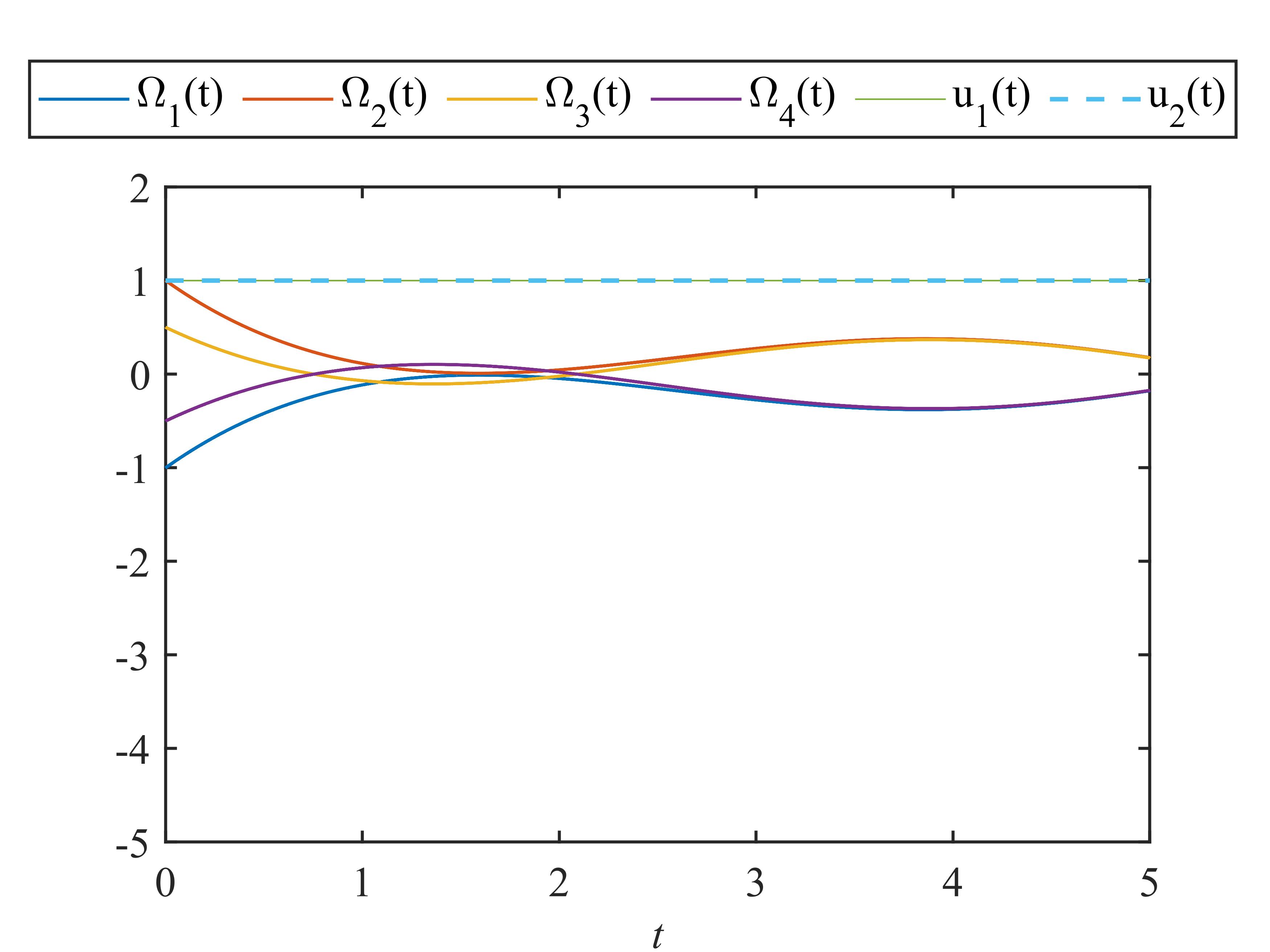}
}
\caption{Global Numerical Solution of Two Stage Drive
System}\label{fig:2sd}
\end{figure}

Further, we can reestablish an equivalent {\DAE} system of this {\IDAE} system with the angle as the variable, which can help us to obtain the exact solutions by symbolic computation.
\begin{eqnarray*}
  \Omega_{1}(t) &=& +\Omega_{2}(t)=-\frac{1}{2}\cdot\cos(t)+C_1\cdot t+C_2\\
  \Omega_{1}(t)&=& -\Omega_{2}(t)=-\frac{1}{4}\cdot(\sin(t)+\cos(t))+C_3\cdot\exp{(-t)}\\
   \Omega_{3}(t) &=& +\Omega_{4}(t)=-\frac{1}{2}\cdot\cos(t)+C_4\cdot t+C_5 \\
  \Omega_{3}(t)&=& -\Omega_{4}(t)=-\frac{1}{4}\cdot(\sin(t)+\cos(t))+C_6\cdot\exp{(-t)}
 \end{eqnarray*}

Here $C_1$, $C_2$, $C_3$, $C_4$, $C_5$, and $C_6$ are constants depending on consistent initial conditions. These exact solutions can be used to check the correctness of our global numerical solution of the embedding method.

It should be noted  that since the global numerical solution adopts the piecewise integration method, the constants $\bm{\xi}$ in the embedding method need to be reassigned along with the integration segment to ensure the consistency of the initial value and the correctness of the solution.

\section{Integro-Partial-Differential-Algebraic Equations}\label{sec:IPDAE}
%In this section, we briefly consider Integro-Partial-Differential-Algebraic Equations ({\IPDAE}).

In this section, we briefly generalize the $\Sigma$-method to {\IPDAE}s based on the modified signature matrix.

%The approach of our paper considers systems of {\IDAE}. In particular if $D_t u(t) = f(t)$, then we can consider $D_t^{-1} (f)$ to be the set of anti-derivatives of $f(t)$.  Further $D_t^{-1}$ is a linear operator.  We may fix the value of $D_t^{-1}$ by setting $D_t^{-1} = \int_{t_0}^t$. In this way we considered a system of {\IDAE} rewritten with respect to these formal inverse operators. The approach developed in this paper is an efficient way to identify the missing constraints for such {\IDAE} under certain conditions.  Crucially {\IDAE} arise naturally in applications.  Theoretically, they are fundamentally important since they also allow the smoothness conditions to be relaxed and solutions of a much broader class considered (existence, uniqueness, well-posedness and construction).

Wu, Reid and Ilie \cite{WU2009923} gave an approach for
the class of so-called $t$-dominated (see Definition \ref{de:t-dominated}) {\PDAE}. This approach can control the growth of differentiation due to only differentiation with respect to a single independent variable $t$. And it can avoid the expensive application of differential elimination methods which are poorly suited to systems with approximate coefficients. 

 The approach developed in this paper is an efficient way to identify the missing constraints for {\IDAE}s under certain conditions. Just like the promotion of {\DAE}s to {\IDAE}s,
we can now consider the class of $t$-dominated {\IPDAE}s 
which involve $D_t^{-1} = \int_{t_0}^t$, with leading $t$ structure dominated by $t$ derivatives and expressions 
in $D_t^{-1}$, which allows the direct application of the previous methods of our paper.

%We now briefly consider generalizations for {\PDAE}.
For simplicity, we only consider the case of $2$ independent variables $(t, u)$. 
The given system we assume to be 
written formally in terms of partial differential operators 
$ D_t = \frac{\partial}{\partial t} $ and $ D_u =\frac{\partial}{\partial u} $ so that it is a 
{\PDAE}.

Consider a set of indeterminates $\Omega=\{v_{\bm{\alpha}}^{j}|\bm{\alpha}=(\alpha_1,\alpha_2) \in \mathbb{N}^2,  j=1,\cdots, n\}$ where each member of $\Omega$ corresponds to a partial derivative by:
$v_{\bm\alpha}^{j}\leftrightarrow (D_t)^{\alpha_1}(D_u)^{\alpha_2}x_j(t,u)$.
In a similar manner to \cite{WU2009923}, we let
$\epsilon > 0$ be a positive symbolic parameter, and define a weight map $\tau:\Omega\rightarrow \R$ with respect to $t$ by
$$\tau(v_{\alpha}^{j}):=\left\{\begin{array}{ll}
    \alpha_{1} & ~~~~if~~\alpha_{2}=0,  \\
    \alpha_{1}+\epsilon & ~~~~otherwise.
\end{array}\right.$$

Similarly to {\IDAE} we can also split an {\IPDAE} $\bm{F}$ into two parts: $\bm{\Phi}$ without integral terms and $\bm{\Psi}$ with integral terms. Then Definition \ref{de:sm_idae} for the signature matrix of an  {\IDAE} can be generalized to {\IPDAE}.
\begin{define}[Signature Matrix for {\IPDAE}] \label{de:sm_ipdae}
  The $n \times n$ signature matrix $\bm \sigma(\bm{F})=[\sigma_{i,j}]_{1 \leq i \leq n,1 \leq j \leq n}(\bm{F})$ of {\IPDAE} $\bm{F}$ of Equation (\ref{eqn:IDAE}) is defined as:
\begin{equation*}
    [\sigma_{i,j}](\bm{F}):= \max_{i,j}{\left([\sigma_{i,j}](\bm{\Phi}),[\sigma_{i,j}](\bm{\Psi})\right)}
\end{equation*}
Here, the definition of $[\sigma_{i,j}](\bm{\Phi})$ and $[\sigma_{i,j}](\bm{\Psi})$ is in the way of Definition \ref{def:LD} and Definition \ref{de:sm_idae}. In which, $x_j(t)$ and $x_j(s)$ should be replaced by $x_j(t,u)$ and $x_j(s,u)$, respectively; the order of $\LD(\cdot,x_j(t))$ and the order of $\LD(\cdot,x_j(s))$ should be replaced by $\tau(\LD(\cdot,x_j(t,u)))$ and $\tau(\LD(\cdot,x_j(s,u)))$, respectively.
\end{define}

\begin{example}\label{ex:IPDAE}
    Let an {\IPDAE} $\bm{F}=\{\frac{\partial^2}{\partial{t}\partial{u}}{x_1(t,u)}+ \int_{t_0}^{t} (s-t)\cdot\frac{\partial^2}{\partial{t^2}}{x_1(t,u)} \cdot x_2(s,u) ds =0, \frac{\partial^2}{\partial{u^2}}{x_2(t,u)}+ \int_{t_0}^{t} (\frac{\partial}{\partial{s}}{x_1(s,u)})^2 ds =0\}$ with independent variables $(t, u)$ and dependent variables $x_1(t,u)$ and $x_2(t,u)$. 
    
    Thus, we can get the following two parts: $\bm{\Phi}=\{\frac{\partial^2}{\partial{t}\partial{u}}{x_1(t,u)},\frac{\partial^2}{\partial{u^2}}{x_2(t,u)}\}$, $\bm{\Psi}=\{ \int_{t_0}^{t} (s-t)\cdot\frac{\partial^2}{\partial{t^2}}{x_1(t,u)} \cdot x_2(s,u) ds, \int_{t_0}^{t} (\frac{\partial}{\partial{s}}{x_1(s,u)})^2 ds \}$

    By Definition \ref{def:LD} and Definition \ref{de:sm_idae}, the $2\times2$ signature matrix (with respect to t) of each parts are:
 $[\sigma_{i,j}](\bm{\Phi})=\left(\begin{array}{cc}
   1+\epsilon & -\infty\\
   -\infty & \epsilon
 \end{array}\right)
 $ and  $[\sigma_{i,j}](\bm{\Psi})=\left(\begin{array}{cc}
   \max(2,-\infty) & -2\\
   \epsilon-1 & 0
 \end{array}\right)
 $.
 
By Definition \ref{de:sm_ipdae}, the $2\times2$ signature matrix (with respect to t) of this {\IPDAE} is:
$[\sigma_{i,j}](\bm{F})=\left(\begin{array}{cc}
   2 & -2\\
   \epsilon-1 & \epsilon
 \end{array}\right)
 $.
    
\end{example}

As in \cite{WU2009923}, $t$-dominated {\PDAE} is dominated by pure derivatives in the independent variable $t$. And a pure derivative of dependent variable $x_j$ to the independent variable $t$ is a derivative form $(\frac{\partial}{\partial t})^{k} x_j(t,u)$ where $k\in \mathbb{N}$. 

\begin{define}\label{de:t-dominated}
We say an {\IPDAE} $\bm{F}$ is dominated by pure derivatives in the independent variable $t$ if there is no $\epsilon$ appearing in $[\sigma_{i,j}](\bm{F})$.
\end{define}

\begin{example}
    Consider the Example \ref{ex:IPDAE}, by Definition \ref{de:t-dominated}, we can say the first equation $F_1$ is t-dominated, and the second equation $F_2$ is not t-dominated since there is $\epsilon$ in its signature matrix.
\end{example}

For example, consider a curtain made of many pendula hanging under gravity $g$ given 
by Wu, Reid and Ilie \cite{WU2009923}.
As shown in Fig. \ref{fig:PendCurtain} the pendula are restricted to move on the surface of the cylinder and in planes
perpendicular to the $s$-axis displayed. The pendula form a continuous curtain in the limit. For
small deviations from the vertical equilibrium position the equations for $x(t, s), y(t, s)$ and Lagrange
multiplier $\lambda(t, s)$ for the continuous curtain satisfy
\begin{eqnarray*}
   \frac{\partial^{2}{x}}{\partial{t^{2}}} + \lambda x = \kappa \frac{\partial^{2}{x}}{\partial{s^{2}}}, \hskip6pt
    \frac{\partial^{2}{y}}{\partial{t^{2}}} + \lambda y + g = \kappa \frac{\partial^{2}{y}}{\partial{s^{2}}},  \hskip6pt 
  x^2 + x^2 = 1.
  \end{eqnarray*}

\begin{figure}
  % Requires \usepackage{graphicx}
  \centering
  \includegraphics[height=3.5cm,width=8cm]{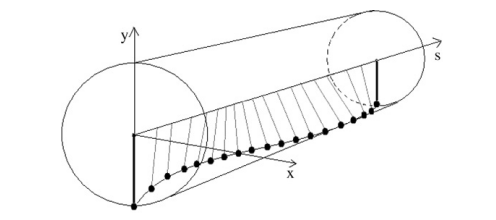}
   \caption{Pendulum Curtain}
  \label{fig:PendCurtain}
\end{figure}

Pishbin's paper \cite{Pishbin2015} contains interesting theoretical and computational discussions of {\DAE} and {\IAE}. In a manner to the pendulum example in 
\cite{Pishbin2015}, the
{\PDAE} for the pendulum curtain can be easily converted to an {\IPDAE} by integrating:
\begin{equation*}
\begin{split}
   \bm{F}=& \left\{  \frac{\partial{x}}{\partial{t}}- \left.{\frac{\partial{x}}{\partial{t}}}\right|_{t=t_0}+\int_{t_0}^{t}\lambda x dt = \int_{t_0}^{t}\kappa  \frac{\partial^2{x}}{\partial{s^2}}dt,  \right. \\
    &\frac{\partial{y}}{\partial{t}}-\left.{\frac{\partial{y}}{\partial{t}}}\right|_{t=t_0}+\int_{t_0}^{t}\lambda y dt  + g =\int_{t_0}^{t}\kappa  \frac{\partial^2{y}}{\partial{s^2}}dt, \\
   & \left. \int_{t_0}^{t} (x^2 + y^2 - 1) dt = 0.\right\}  
\end{split}
\end{equation*}

The key idea is to select an independent variable and ranking of derivatives in which the derivatives are highest in the ranking. In this example, the ranking for each dependent variables, {\eg} $x$, should satisfy $ x\prec \frac{\partial{x}}{\partial{s}} \prec \frac{\partial^{s}{x}}{\partial{s^2}}\prec \cdots\prec\frac{\partial{x}}{\partial{t}}\prec \frac{\partial^{2}{x}}{\partial{t}\partial{s}}\prec  \cdots$.  This system is $t$-dominated and a generalization of the $\Sigma$-method applies.

Finally, we can get its signature matrix of this {\IPDAE} as:
$$
 [\sigma_{i,j}](\bm{F})=\left(\begin{array}{ccc}
   ~1 & ~-\infty & ~-1\\
     ~-\infty & ~1 & ~-1\\
     -1 & -1 & ~-\infty
 \end{array}\right).
$$
%For more detail, we take the first equation and dependent variable $x$ as an example, by Definition \ref{de:sm_ipdae}, $[\sigma_{1,1}](\bm{F})=\max(1,\max(\epsilon,\epsilon-\infty)))=1$. 

By Definition \ref{de:t-dominated}, $\bm{F}$ is $t$-dominated.

In the manner to {\PDAE}s, we can apply the $\Sigma$-method to this $t$-dominated system directly, and get $\bm{c}=(1,1,3)$, $\bm{d}=(2,2,0)$. Then, we can easily yield the $t$-differentiation of this {\IPDAE}:
$\bm{F}^{(\bm c)}=\{\{
     \frac{\partial^{2}{x}}{\partial{t^{2}}} + \lambda x = \kappa \frac{\partial^{2}{x}}{\partial{s^{2}}},\hskip6pt
     \frac{\partial^{2}{y}}{\partial{t^{2}}} + \lambda y + g = \kappa \frac{\partial^{2}{y}}{\partial{s^{2}}},\hskip6pt x\frac{\partial^{2}{x}}{\partial{t^{2}}}+y\frac{\partial^{2}{y}}{\partial{t^{2}}}+(\frac{\partial{x}}{\partial{t}})^2+(\frac{\partial{y}}{\partial{t}})^2=0
\},\hskip6pt\{\frac{\partial{x}}{\partial{t}}- \left.{\frac{\partial{x}}{\partial{t}}}\right|_{t=t_0}+\int_{t_0}^{t}\lambda x dt = \int_{t_0}^{t}\kappa  \frac{\partial^2{x}}{\partial{s^2}}dt,\hskip6pt \frac{\partial{y}}{\partial{t}}-\left.{\frac{\partial{y}}{\partial{t}}}\right|_{t=t_0}+\int_{t_0}^{t}\lambda y dt  + g =\int_{t_0}^{t}\kappa  \frac{\partial^2{y}}{\partial{s^2}}dt,\hskip6pt x\frac{\partial{x}}{\partial{t}}+y\frac{\partial{y}}{\partial{t}} =0\},\hskip6pt\{x^2 + x^2 = 1\},\hskip6pt\{\int_{t_0}^{t} (x^2 + y^2 - 1) dt = 0\}\}
$.

Obviously, the Jacobian matrix of top block is 
$$\Jac=\left(\begin{array}{ccc}
   1 & 0 & x\\
    0 & 1 & y\\
     x & y & 0
 \end{array}\right),$$
 whose determinant is $x^2+y^2$ is non-singular satisfying its constraints. 

 Due to there is no integral item in the top block of the differentiation system $\bm{F}^{(\bm c)}$, the top block of $\bm{F}^{(\bm c)}$ is a typical {\PDAE}. By the Theorem $7.3$ of \cite{WU2009923}, since the Jacobian matrix of top block of $\bm{F}^{(\bm c)}$ is non-singular, the top block can be transformed into a {\DAE} system via the numerical method of lines \cite{WU2009923} and the Jacobian matrix of this {\DAE} is non-singular too. That's to say we can numerically solve this example after numerical discretization. 
 
However, this example is only of constant coefficient. Considering the length of this paper, numerical solution based on our approach in this paper for a general {\IPDAE} is our future work.

 \section{Conclusions}\label{sec:con}

There are many obstacles to the structural analysis by the $\Sigma$-method  for {\IDAE}s.To clear it, we redefine the signature matrix, so that it can deal with general forms of {\IDAE}s. 

However, the $\Sigma$-method may fail due to overestimating some elements in signature matrix including incorrect signature matrix cases and S-unamenable cases. To correct the signature matrix, we give an efficient detection method by points that also helps to calculate the rank of the Jacobian. To regularize S-unamenable {\IDAE}s, we remedy the condition of convergence and termination of the embedding method with a new definition of the {\DOF} for {\IDAE}s. The embedding method for an {\IDAE} avoids direct elimination by introducing new variables and equations to increase the dimensions of space in which the {\IDAE} resides. Under certain conditions, it avoids solving assignment problems for the new systems. The superiority of the embedding method has been illustrated with examples.

 For initial points, we can traverse all components by the Homotopy methods and interval methods. Combined with the embedding method, the global numerical method can find all numerical solutions of {\IDAE}s. We give an example of two stage drive system to demonstrate the method.

As shown in example of pendulum curtain, it is promising that modified signature matrix can help to transform an {\IPDAE} system into an {\IDAE} system. A general numerical solution method for {\IPDAE}s is our future work.

\appendix

\medskip\noindent{\bf \normalsize Acknowledgements.}
This work is partially supported by the projects of Chongqing (2020000036, 2021000263, cstc2020yszx-jcyjX0005, Chongqing Talents - Wenyuan Wu), and special research assistant program of CAS. 

\bibliographystyle{elsarticle-num-names}
\bibliography{cas-refs}

%% else use the following coding to input the bibitems directly in the
%% TeX file.

% \begin{thebibliography}{00}

% %% \bibitem[Author(year)]{label}
% %% Text of bibliographic item

% \bibitem[ ()]{}

% \end{thebibliography}
\end{document}